\documentclass[english]{amsart}
\usepackage[T1]{fontenc}
\usepackage[latin9]{inputenc}
\usepackage{color}
\usepackage{babel}
\usepackage{amsbsy}
\usepackage{amstext}
\usepackage{amsthm}
\usepackage{amssymb}
\usepackage[unicode=true,pdfusetitle,
 bookmarks=true,bookmarksnumbered=false,bookmarksopen=false,
 breaklinks=false,pdfborder={0 0 0},pdfborderstyle={},backref=false,colorlinks=true]
 {hyperref}
\hypersetup{
 linkcolor=brown, citecolor=blue, pdfstartview={FitH}, hyperfootnotes=false, unicode=true}

\makeatletter
\numberwithin{equation}{section}
\numberwithin{figure}{section}
\theoremstyle{plain}
\newtheorem{thm}{\protect\theoremname}
\theoremstyle{definition}
\newtheorem{defn}[thm]{\protect\definitionname}
\theoremstyle{plain}
\newtheorem{lem}[thm]{\protect\lemmaname}
\theoremstyle{definition}
\newtheorem{problem}[thm]{\protect\problemname}
\theoremstyle{plain}
\newtheorem{prop}[thm]{\protect\propositionname}
\theoremstyle{definition}
\newtheorem{example}[thm]{\protect\examplename}
\theoremstyle{remark}
\newtheorem{rem}[thm]{\protect\remarkname}
\theoremstyle{plain}
\newtheorem{cor}[thm]{\protect\corollaryname}

\usepackage{tikz-cd}
\usepackage{bbm}

\makeatother

\providecommand{\corollaryname}{Corollary}
\providecommand{\definitionname}{Definition}
\providecommand{\examplename}{Example}
\providecommand{\lemmaname}{Lemma}
\providecommand{\problemname}{Problem}
\providecommand{\propositionname}{Proposition}
\providecommand{\remarkname}{Remark}
\providecommand{\theoremname}{Theorem}

\begin{document}
\title[Markov conditions for continuity of computable functions ]{A generalization of Markov\textquoteright s approach to the continuity
problem for Type 1 computable functions }
\author{Emmanuel Rauzy }
\thanks{emmanuel.rauzy.14@normalesup.org, Universität der Bundeswehr München.
The author is funded by an Alexander von Humboldt Research Fellowship.}
\begin{abstract}
We axiomatize and generalize Markov\textquoteright s approach to the
continuity problem for Type 1 computable functions, i.e. the problem
of finding sufficient conditions on a computable topological space
to obtain a theorem of the form « computable functions are (effectively)
continuous ». We introduce different notions of \emph{effective closure.
}These notions of effective closure lead to different notions of \emph{effective
discontinuity} \emph{at a point}. We give conditions that prevent
computable functions from having effective discontinuities. We finally
show that results that forbid effective discontinuities can be turned
into (abstract) continuity results on spaces where the closure and
effective closure of semi-decidable sets naturally coincide -this
happens for instance on spaces which admit a dense and computable
sequence. 
\end{abstract}

\maketitle

\section{\label{sec:Introduction}Introduction}

In his 1937 seminal paper \cite{Turing1937}, and in the subsequent
corrigendum \cite{Turing1938}, Alan Turing gave a possible definition
for what it means for a real function to be computable. 

His definition is as follows. 

One first fixes a preferred description for real numbers: in this
case, a real number is given by a Turing machine that computes rational
approximations of this number to any desired precision. 

One can then define a computable function of a real variable to be
a function for which it is possible, given the description of a real
number, to compute a description of its image. 

This notion of computability is now known as Type 1 computability.
The reals that admit descriptions via Turing machine as above are
called the computable reals, the set of computable reals is denoted
$\mathbb{R}_{c}$. 

A remarkable result in this setting is that a total Type 1 computable
function $f:\mathbb{R}_{c}\rightarrow\mathbb{R}_{c}$ must be continuous,
and even \emph{effectively continuous}. 

What is now known as the ``continuity problem'' \cite{SPREEN_2016}
is the question of to what extent can this result be generalized:
what are the weakest conditions that enforce computable functions
between computable topological spaces to always be continuous? and
effectively continuous? One must distinguish between conditions on
domains and conditions on codomains, those are in general very different.
Results which partially address this problem are called ``continuity
results'' or ``continuity theorems''. 

Just as the unsolvability of the halting problem and the Rice and
Rice-Shapiro theorems can be understood as showing that ``information
on the behavior of a Turing machine can be obtained only by running
this machine'', i.e. via the UTM theorem, continuity theorems on
$\mathbb{R}_{c}$ tend to show that information on a real number described
by a Turing machine that produces rational approximation of it can
be obtained only by running the given machine, and using the produced
approximations. 

Note that continuity results do not relativize: the function $\delta_{0}:\mathbb{R}_{c}\rightarrow\mathbb{R}_{c}$,
which maps $0$ to $1$ and all other computable reals to $0$, is
computable modulo the halting problem. This differs from what happens
in the Type 2 theory of effectivity, which is the modern approach
to computable analysis \cite{Brattka2021}: the functions that are
computable with respect to the Cauchy representation of the real numbers
are continuous, and this remain true no matter what oracle we allow.
Even more: a function $f:\mathbb{R}\rightarrow\mathbb{R}$ is continuous
if and only if it is Type 2 computable modulo some oracle. 

\bigskip

Historically, the first continuity result is probably that of Banach
and Mazur \cite{Banach1937}, which is accounted for in \cite{Mazur1963}. 

The notion of ``computable function'' considered by Banach and Mazur
is not without flaws (a Banach-Mazur computable function is not attached
to a finite description, and, on peculiar domains, there can exist
continuously many Banach-Mazur computable functions). However, as
Hertling has showed in \cite{Hertling2002} that an effective continuity
theorem cannot be proved for Banach-Mazur computable functions, Banach-Mazur
computability can now be seen as a kind of reverse mathematics tool:
an effective continuity theorem cannot follow from statements that
are true of Banach-Mazur computability. We discuss this in Section
\ref{subsec:Banach-Mazur-computability}.

Follow the non-effective discontinuity result of Markov \cite{Markov1963},
the effective continuity theorem of Kreisel-Lacombe-Schoenfield \cite{Kreisel1957}
on Baire space, the effective continuity theorem of Ceitin on any
computable Polish space \cite{Ceitin1967}, which is accounted for
in Kushner \cite{Kushner1984}. This last result of Ceitin is usually
known as the Kreisel-Lacombe-Schoenfield-Ceitin theorem, or KLST\footnote{Ceitin is also often written Tseitin}
Theorem for short. 

At approximately the same time appears the article of Moschovakis
``Recursive metric spaces'' \cite{Moschovakis1964}. We want to
note here that the result of Moschovakis from \cite{Moschovakis1964}
is often misquoted. Indeed, there is proved a result strictly stronger
than Ceitin's, replacing the assumption of existence of a computable
and dense sequence by an \emph{effective Axiom of Choice} condition.
This effective axiom of choice takes the form of an algorithm that,
given a semi-decidable set and an open ball which it intersects, produces
a point in this intersection. In that same article, Moschovakis proved
that two a priori different notions of ``effectively open set''
agree on recursive Polish spaces: the Lacombe sets, computable unions
of basic open sets, are exactly the semi-decidable sets for which
it is possible, given a point of the set, to produce an open ball
around this point still contained in the original set. This result
is discussed in the author's paper \cite{RAUZY2023}. 

In 1912, in \cite{Borel1912}, Borel suggested that although it seems
clear that being able to compute arbitrarily good approximations of
a number does not allow to recognize this number, one might attempt
to use \emph{the size of the definition one has access to} to bridge
the gap between obtaining approximations, that could correspond to
infinitely many reals, and actually recognizing a number. In \cite{Hoyrup2016},
Hoyrup and Rojas formalized Borel's remark thanks to Kolmogorov complexity,
and showed that the proposed method, while insufficient to define
discontinuous computable functions, does allow to define non open
semi-decidable sets. This sheds new light on a well known example
of Friedberg \cite{Friedberg1958}. This approach even provides a
characterization of the Ershov topology on a set consisting of a single
converging sequence together with its limit: in this case, the non-open
semi-decidable sets all arise from Borel's proposed method. Is also
constructed there a semi-decidable set which is not $\Sigma_{2}^{0}$
in the classical Borel hierarchy, thus higher than Friedberg's example. 

One could say that this approach gives the most modern point of view
on the continuity problem. We will however not follow it here. 

In \cite{Spreen1984}, in a joint work with Paul Young, and in subsequent
articles \cite{Spr98,Spreen_2010,SPREEN_2016}, Dieter Spreen gave
an axiomatic approach to the continuity problem in topological spaces
in terms of ``witness of non-inclusion''. A computable function
between computable topological spaces has a witness for non-inclusion
if there is a procedure which can, whenever the preimage of an open
set in the codomain is not included in an open set of the domain,
produce a point in the difference (in fact one is allowed to change
slightly the open set of the domain to a smaller semi-decidable set).
In \cite{Spreen1984} it was shown that this condition automatically
implies effective continuity of the considered function, assuming
that computing limits is possible on the domain of this map. In \cite{Spr98},
the point of view is shifted: one does not ask anymore whether a function
has a witness for non-inclusion, but whether two effective topologies
on the same numbered set have a realizer for non-inclusion with respect
to one another, i.e. a program that will produce a point separati
ng two open sets when one is not a subset of the other. Spreen then
proves that an effective topology which has a realizer for non-inclusion
with respect to another effective topology is effectively finer than
the latter (this also holds only if computing limits is possible).
One then recovers continuity results for sets of functions by considering
their initial topologies. This allows to recover theorems of Ceitin
\cite{Ceitin1967}, Kreisel-Lacombe-Schoenfield \cite{Kreisel1957},
Myhill and Sheperdson \cite{Myhill1955} in a uniform manner. 

Finally, in \cite{Hertling1996}, Peter Hertling gave necessary and
sufficient conditions for a computable function to be effectively
continuous. While this does not solve the continuity problem -because
the given conditions depend on the function and not only on the space
on which it is defined- this result provides a good understanding
of the continuity problem. The result proven in \cite{Hertling1996}
says that a function $f$ defined on a subset of a computably separable
computable metric space is continuous if and only if for each $\epsilon>0$,
there is a dense and computable sequence, not in the domain of $f$,
but where $f$ ``seems to be defined, up to $\epsilon$''. (And
the dependence of this sequence in $\epsilon$ should be computable.)
See \cite{Hertling1996} for more details and for proper statements. 

\bigskip

Note that the above results are all set in Markovian computable analysis.
Most of these results have counterparts in various forms of constructive
mathematics. In \cite{Beeson1976}, Beeson gave a proof of the KLST
Theorem in the system \textbf{HA} of Heyting Arithmetic with Markov's
principle (a binary sequence must contain a $1$ or be identically
null). Ishihara gave in \cite{Ishihara1991,Ishihara1992} constructive
proofs of the KLST theorem and of various weaker continuity theorems,
and furthermore he pinpointed various principles that are logically
equivalent to these continuity results, thus precisely classifying
them in terms of constructive reverse mathematics. 

Recently, Bauer \cite{Bauer2023} formalized in synthetic topology
Spreen's generalization of the KLST theorem. Even if one were interested
solely in Markov computability, this paper is very interesting, because
it provides a very concise proof of Spreen's result which highlights
each of its necessary ingredient. We will refer ourselves to several
concepts introduced in \cite{Bauer2023}, in particular that of a
\emph{intrinsic subset}. 

\bigskip

Our purpose here is to give very simple conditions on a computable
topological space that impose continuity of computable functions.
As Spreen, we follow an axiomatic approach, giving natural conditions
on a space that impose continuity of computable functions, but the
conditions we propose are different from his, and should thus provide
a new understanding of the continuity problem. 

Our approach is in fact very much inspired by Markov's article \cite{Markov1963}.
In this article, it is shown that on the set $\mathbb{R}_{c}$ of
computable reals, the limit of a computable sequence cannot be told
apart from the elements of this sequence \cite[Theorem 4.2.2]{Markov1963},
and that this preliminary result prevents functions defined on $\mathbb{R}_{c}$
from having effective discontinuities \cite[Theorem 5.1]{Markov1963}.
We axiomatize Markov's result: we say that a space satisfies a \emph{Markov
condition} if, whenever a point $y$ is effectively adherent to a
set $A$, then $\{y\}$ is not a semi-decidable subset of $\{y\}\cup A$.
We then show that computable functions defined on a space satisfying
a Markov condition cannot have effective discontinuities. This is
thus Theorem 3 of \cite{Markov1963} in its most general setting.

Finally, we note that for a function whose domain is a spaces where
closure and effective closure of semi-decidable sets naturally agree
and whose codomain is a space where points have neighborhood bases
of co-semi-decidable sets, this ``no effective discontinuities''
result turns into an abstract continuity result. The property ``closure
and effective closure of semi-decidable sets agree'' appears as a
natural generalization of effective separability, generalization which
plays a crucial role for continuity results about Markov computable
functions. 

The results we present, while not implied by already known results,
are not strictly more general either. In particular we do not prove
effective continuity results. However, we provide easily proven results
which are aimed at providing a better understanding of the continuity
problem. We introduce new definitions (effective closure, effective
discontinuities, Markov conditions, multi-numberings of neighborhoods,
effective neighborhood bases) which seem to us relevant to the continuity
problem and which had previously not been stated in the degree of
generality we consider. For instance, it is often supposed that every
point has a computably enumerable basis of neighborhood. In this case,
the distinction between the \emph{effective closure} and the \emph{sequential
effective closure} becomes irrelevant, and ``effective discontinuities''
can be defined in terms of sequences. But this is not the case in
general, and a general definition of an ``effective discontinuity''
had never been written out before. 

Thus the present article is only supplemental to the important articles
of Ceitin \cite{Ceitin1967} and Moschovakis \cite{Moschovakis1964}
and the account of their results by Kushner \cite{Kushner1984}, and
to the more recent work of Hertling \cite{Hertling1996}, Spreen \cite{Spreen1990,Spr98,Spreen_2010,SPREEN_2016},
Hoyrup and Rojas \cite{Hoyrup2016}. 

The notion of a computable topological space we consider is that introduced
in the author's paper \cite{RAUZY2023}, it is the most general notion
currently available, as it was not obtained by effectivizing a notion
of basis, but by directly effectivizing the notion of topological
space. The same definition was introduced independently by Bauer in
\cite{Bauer2023} in the context of synthetic topology (and thus it
is only the same definition modulo translation). A \emph{Type 1} \emph{computable
topological space} is a quadruple $(X,\nu,\mathcal{T},\tau)$, where:
\begin{itemize}
\item $(X,\mathcal{T})$ is a classical topological space. 
\item $\nu$ is a numbering of $X$. 
\item $\tau$ is a numbering of a subset of $\mathcal{T}$ whose image generates
$\mathcal{T}$ when seen as a (classical) basis. Open sets in the
image of $\tau$ are the \emph{effective open sets, }they are asked
to be uniformly semi-decidable sets. 
\item The operations of taking finite intersection and computable unions
should be computable for $\tau$. 
\end{itemize}
This is detailed in Section \ref{part: Preliminaries}. 

It is clear that with this broad definition of Type 1 computable topological
space, computable functions between computable topological spaces
do not have to be continuous. In other words, a computable topology
on a numbered set can be very far from the Ershov topology, the topology
of semi-decidable sets. For instance, every topology on a finite set
(equipped with its natural decidable numbering) can be rendered effective,
and since functions between finite sets with decidable numberings
are all computable, we obtain many computable functions that are not
continuous. 

\bigskip

We start by defining the effective closure of a subset $A$ of a computable
topological space $(X,\nu,\mathcal{T},\tau)$: a point $x$ is \emph{effectively
adherent} to $A$ if there is a program that, on input the $\tau$-name
of an open set $O$ that contains $x$, produces the $\nu$-name of
a point in $A\cap O$. We denote this by $x\in\overline{A}^{+}$.
In the text we also define two other notions of effective closure:
\emph{effective sequential closure} and \emph{normed effective sequential
closure}. 

We can then define effective discontinuities. Let $f:X\rightarrow Y$
be a computable function between computable topological spaces $(X,\nu,\mathcal{T}_{1},\tau_{1})$
and $(Y,\mu,\mathcal{T}_{2},\tau_{2})$. An \emph{effective discontinuity}
at a point $x$ is an effective open set $O_{2}\in\mathcal{T}_{2}$
which contains $f(x)$, but such that $x$ belongs to the effective
closure of $f^{-1}(O_{2})^{c}$, the complement in $X$ of $f^{-1}(O_{2})$.
Note that each different notion of effective closure gives rise to
a notion of effective discontinuity. 

We can now state our main definition, which we name after Markov,
as it is underlying the article \cite{Markov1963}. (With our vocabulary,
Theorem 4.2.2 of \cite{Markov1963} thus reads: $\mathbb{R}_{c}$
satisfies a Markov condition with respect to the effective sequential
closure.)
\begin{defn}
Let $(X,\nu,\mathcal{T},\tau)$ be a computable topological space.
We say that it satisfies \emph{a Markov condition with respect to
effective closure }if for any $x$ in $X$ and any subset $A$ of
$X$ that does not contain $x$, if $x$ is effectively adherent to
$A$, then $\{x\}$ is not a semi-decidable subset of $A\cup\{x\}$. 
\end{defn}

Again, in the text, we also introduce the Markov condition with respect
to the effective sequential closure and to the normed effective sequential
closure. We note in Theorem \ref{thm: Comp Sober =00003D> Markov cond Normed }
that any computably sober space satisfies the Markov conditions with
respect to the normed effective sequential closure. 

In any numbered set $(X,\nu)$, denote by $\rightsquigarrow_{\nu}$
the relation on $\mathcal{P}(X)$ given by: 
\[
A\rightsquigarrow_{\nu}B\iff B\text{ is not a \ensuremath{\nu}-semi-decidable subset of }A\cup B.
\]

The definition above is then summed up as: 
\[
\forall x\in X,\,\forall A\subseteq X,\,x\notin A\,\&\,x\in\overline{A}^{+}\implies A\rightsquigarrow_{\nu}\{x\}.
\]

This condition seems to us to express in a very precise way the link
between the computable topology on $X$ and properties of $\nu$ that
can enforce continuity of computable functions. Indeed, the implication
written above renders explicit the control that the topology has on
the numbering $\nu$: the statement $A\rightsquigarrow_{\nu}\{x\}$,
which does not make reference to the topology, is governed by the
statement $x\in\overline{A}^{+}$. Markov conditions imply that \emph{if
a point $x$ is effectively adherent to a set $A$,} (which implies
in particular that it cannot be separated from $A$ by an open set,
this is a topological condition), \emph{then it cannot be separated
from $A$ by a semi-decidable set }(condition on $\nu$ that is independent
of the topology)\emph{.}

The first lemma we present is: 
\begin{lem}
\label{lem: No eff disc 1-1}Suppose that $(X,\nu,\mathcal{T}_{1},\tau_{1})$
is a computable topological space that satisfies a Markov condition\emph{
}with respect to effective closure, and that $(Y,\mu,\mathcal{T}_{2},\tau_{2})$
is any computable topological space.

Then $(\nu,\mu)$-computable functions between $X$ and $Y$ cannot
have effective discontinuities with respect to $(\mathcal{T}_{1},\tau_{1})$
and $(\mathcal{T}_{2},\tau_{2})$. 
\end{lem}

We then introduce identification conditions between the closure and
effective closure of semi-decidable sets, that will allow us to transform
discontinuities into effective discontinuities. A computable topological
space $(X,\nu,\mathcal{T},\tau)$ satisfies an \emph{identification
condition between the closure and the effective closure for semi-decidable
sets }if for any $\nu$-semi-decidable set $A$, the closure and effective
closure of $A$ agree, i.e. $\overline{A}=\overline{A}^{+}$. 

In terms of the constructivist logic used in the work of Bauer \cite{Bauer2000},
this statement can be expressed as the reduction of a double negation:
\[
\neg\neg(x\in\overline{A})\implies x\in\overline{A}.
\]
Note that here the fact that $x\in\overline{A}$ is taken as the statement:
\[
\forall O\in\mathcal{T},\,x\in O\implies\exists a\in A\cap O,
\]
and the effective version of this statement is exactly what we denote
by $x\in\overline{A}^{+}$.

We note in Proposition \ref{prop: Identification condition co-semi-decidable sets}
that, if we were to replace semi-decidable sets by co-semi-decidable
sets in the definitions above, we would obtain much too restrictive
conditions, and thus ``identification conditions between the closure
and the effective closure for \emph{co-semi-decidable sets}'' are
of no interest. 

To be able to use these identification conditions, we have to introduce
conditions on the codomain of the considered functions. 

A well known example of Friedberg \cite{Friedberg1958} of a semi-decidable
set of $\mathbb{R}_{c}$ which is not open shows that the continuity
theorems for functions between metric spaces cannot be extended to
functions whose codomain is the Type 1 computable Sierpi\'{n}ski space\footnote{This is the set $\{0,1\}$, equipped with the topology generated by
$\{1\}$ and $\{0,1\}$, and with numbering $s$ given by $s(n)=0$
if $n\in K$ and $s(n)=1$ if $n\notin K$, where $K$ designates
the halting set.}. 

The crucial property that computable metric spaces have, and that
fails for the Sierpi\'{n}ski space, is that points have bases of neighborhoods
which are \emph{co-semi-decidable}, instead of semi-decidable as is
in general expected from a computable topological space. The fact
that this condition is central to establishing continuity results
was fist expressed by Dieter Spreen in \cite{Spr98}. 

In the present paper, neighborhood bases of co-semi-decidable sets
are required in order to apply the identification condition for semi-decidable
sets (as we pass from a neighborhood of a point in the codomain of
a function to the complement of its preimage in the domain, the original
neighborhood has to be co-semi-decidable if we want its complement
to be semi-decidable). 

We finally prove:
\begin{thm}
\label{thm:Continuity-result-eff-closure-1}Suppose that $(X,\nu,\mathcal{T}_{1},\tau_{1})$
is a computable topological space that satisfies a Markov condition
for the effective closure, and the identification condition between
the closure and the effective closure for semi-decidable sets. 

Suppose that $(Y,\mu,\mathcal{T}_{2},\tau_{2})$ is a computable topological
space where points have co-semi-decidable bases of neighborhood. 

Then $(\nu,\mu)$-computable functions between $X$ and $Y$ are $(\mathcal{T}_{1},\mathcal{T}_{2})$-continuous. 
\end{thm}

In the text, we obtain two similar theorems, one for the effective
sequential closure, and one for the normed effective sequential closure.

We are not able to prove an effective continuity result based on the
conditions given above. Throughout the article, we make sure to define
the uniform conditions associated to the non-uniform conditions that
appear in Theorem \ref{thm:Continuity-result-eff-closure-1}, to have
a chance to prove an uniform continuity theorem. In Section \ref{subsec:Effective-Continuity-conjecture},
we detail why the effective axiom of choice condition given by Moschovakis
in \cite{Moschovakis1964} is very similar (yet more restrictive)
than the \emph{uniform identification condition between the closure
and effective closure} \emph{for semi-decidable sets, }and ask for
an effective continuity theorem that relies on this latter condition.

\bigskip

The present work is motivated by the study of computability on the
space of marked groups. A marked group is simply a countable group
$G$ together with a tuple of elements $S$ that generate it. Every
element of $G$ can then be seen as a word of $(S\cup S^{-1})^{*}$,
but different words may define the same element. The topology of the
space of marked group is defined by saying that $(G,S=(s_{1},...,s_{k}))$
and $(H,S'=(s'_{1},...,s'_{k}))$ are at distance at most $2^{-n}$
exactly when the words of length at most $n$ in $(S\cup S^{-1})^{*}$
that define the identity of $G$ are the same as the words of length
at most $n$ in $(S'\cup S'^{-1})^{*}$ that define the identity in
$H$, modulo the identification $s_{i}\mapsto s_{i}'$. 

In \cite{Rauzy2021}, the author has shown that many known continuity
theorems cannot be applied to the space of marked groups: it does
not have a computable and dense sequence and does not satisfy Moschovakis'
effective Choice Axiom. However, the results in \cite{Rauzy2021}
do not invalidate the possibility of applying the continuity theorem
presented here to the space of marked groups. 

The space of marked groups being an effectively complete recursive
metric space, it satisfies the Markov Conditions. The only remaining
question concerns the identification conditions. 

We thus ask: 
\begin{problem}
Does the space of marked groups satisfy identification conditions
between closure and effective closure for semi-decidable sets? 
\end{problem}

We are more inclined to thinking that the space of marked groups does
not satisfy such conditions. However, what is particularly interesting
to us is that if a space fails to satisfy these identification conditions,
there must be a point and a semi-decidable set which are the ones
that cause this condition to fail: failure of these conditions is
a localized phenomenon. Indeed, if these conditions fail, there is
a point $x$ and a set $A$ such that $x\in\overline{A}$ but $x\notin\overline{A}^{+}$.
Exhibiting such a pair in the space of marked groups would be very
interesting. On the contrary, the statement ``there is no computable
and dense sequence'' is not localized, we cannot really say ``where''
the set fails to have a dense and computable sequence. 

\subsection*{Acknowledgements }

I thank the anonymous referees for many remarks that led to the improvement
of this paper, in particular for the remark that no codomain conditions
are required for Lemma \ref{lem: No eff disc 1-1}. 

\section{\label{part: Preliminaries}Preliminaries }

\subsection{Subnumberings, multi-subnumberings }

A \emph{numbering} of a set $X$ is a surjection $\nu$ that maps
a subset of the set $\mathbb{N}$ of natural numbers onto $X$. 

For convenience, we introduce the notion of subnumbering. A \emph{subnumbering}
of a set $X$ is a numbering of a subset of $X$. 

A subnumbering can be seen as a numbering of a countable set $A$
together with an embedding of $A$ into a bigger set $X$. For instance,
considering the numbering of computable reals as a subnumbering of
the set of all real numbers indicates that we consider the set of
computable reals embedded in the set of real numbers. 

We denote this by $\nu:\subseteq\mathbb{N}\rightarrow X$. 

Denote by $\text{dom}(\nu)$ the domain of $\nu$. If $\nu$ is a
subnumbering of $X$, we call the image of $\nu$ the set of $\nu$\emph{-computable
points}. This set is denoted $X_{\nu}$. 

Denote by $\mathcal{N}_{X}$ the set of subnumberings of $X$. 

If $\nu(n)=x$, then $n$ is a $\nu$\emph{-name} of $x$. 

A \emph{multi-subnumbering} of $X$ is a map from $\mathbb{N}$ to
the power set of $X$. We denote this by $\nu:\subseteq\mathbb{N}\rightrightarrows X$,
which is just a synonym for $\nu:\mathbb{N}\rightarrow\mathcal{P}(X)$.
Multi-subnumberings are always total, but if $\nu$ is a multi-subnumbering
and $n$ is a number such that $\nu(n)=\emptyset$, then $n$ does
not describe anything, one can think of this as $\nu$ not being defined
at $n$. If $\nu$ is a multi-subnumbering, we denote by $\text{dom}(\nu)$
the set $\{n\in\mathbb{N},\,\nu(n)\ne\emptyset\}$. (Multi-numberings
are well known objects. Multi-representations were first investigated
by Schröder in \cite{Schroeder2003}. Multi-numberings are used for
instance by Spreen in \cite{Spreen_2010}.)

The set of multi-subnumberings on $X$ is denoted $\mathcal{MN}_{X}$.

In case $x\in\nu(n)$, we say that $n$ is a $\nu$\emph{-name }of
$x$. The number $n$ may not characterize $x$, and thus it constitutes
in general a partial description of $x$. The\emph{ set of $\nu$-computable
points} is the set of points that have a $\nu$-name. 
\begin{defn}
\label{def:Comp for multi numberings}A function $f:X\rightarrow Y$
between sets equipped with multi-subnumberings $\nu$ and $\mu$ is
called\emph{ $(\nu,\mu)$-computable} if there exists a partial computable
function $F:\subseteq\mathbb{N}\rightarrow\mathbb{N}$ defined at
least on $\text{dom}(\nu)$ such that $\forall n\in\text{dom}(\nu),\,f(\nu(n))\subseteq\mu(F(n))$. 
\end{defn}

Equivalently: the function $f$ is \emph{$(\nu,\mu)$-}computable
if there is a partial computable function $F:\subseteq\mathbb{N}\rightarrow\mathbb{N}$
such that, for any $x$ in $X$ and $n$ in $\mathbb{N}$, if $n$
is a $\nu$-name of $x$, then $F(n)$ is a $\mu$-name of $f(x)$. 

A subnumbering $\nu$ can be identified with a multi-subnumbering
$\hat{\nu}$ by putting:
\[
\forall n\in\text{dom}(\nu),\,\hat{\nu}(n)=\{\nu(n)\},
\]
\[
\forall n\notin\text{dom}(\nu),\,\hat{\nu}(n)=\emptyset.
\]

Thanks to this identification, Definition \ref{def:Comp for multi numberings}
above thus also provides a definition for computability with respect
to subnumberings. 

If $\nu$ and $\mu$ are two multi-subnumberings of a set $X$, define
$\nu\le\mu$ by 
\[
\nu\le\mu\iff\text{id}_{X}:X\rightarrow X\text{ is \ensuremath{(\nu,\mu)}-computable}.
\]

This means that $\nu$-names of points contain more information than
$\mu$-names. Define $\nu\equiv\mu$ if $\nu\le\mu$ and $\mu\le\nu$. 

\subsection{Product numbering, numbering of functions }

Throughout, we denote by $n,m\mapsto\langle n,m\rangle$ Cantor's
pairing function. 

If $\nu$ is a multi-numbering of $X$ and $\mu$ a multi-numbering
of $Y$, denote by $\nu\times\mu$ the multi-numbering of $X\times Y$
defined by: if $n$ is a $\nu$-name of a point $x$ of $X$, and
$m$ a $\mu$-name of a point $y$ of $Y$, then $\langle n,m\rangle$
is a $\nu\times\mu$-name of $(x,y)$, i.e. 
\[
(x,y)\in(\nu\times\mu)(\langle n,m\rangle)\iff x\in\nu(n)\,\&\,y\in\mu(m).
\]
This is the product multi-numbering. 

We fix a standard enumeration of all partial recursive function, denoted
$\varphi_{0}$, $\varphi_{1}$, $\varphi_{2}$,... We write $\varphi_{n}(k)\downarrow$
if $k\in\text{dom}(\varphi_{n})$ and $\varphi_{n}(k)\uparrow$ if
$k\notin\text{dom}(\varphi_{n})$. 

This allows us to define the numbering associated to computable functions. 
\begin{defn}
If $(X,\nu)$ and $(Y,\mu)$ are numbered sets, we define a subnumbering
$\mu^{\nu}$ of the set $Y^{X}$ of functions from $X$ to $Y$ as
follows: 
\begin{align*}
\text{dom}(\mu^{\nu}) & =\{i\in\mathbb{N}\vert\text{ dom}(\nu)\subseteq\text{dom}(\varphi_{i}),\\
 & \forall n,m\in\text{dom}(\nu),\nu(n)=\nu(m)\implies\mu(\varphi_{i}(n))=\mu(\varphi_{i}(m))\},
\end{align*}
\[
\forall i\in\text{dom}(\mu^{\nu}),\forall x\in X,\forall k\in\text{dom}(\nu),\,(x=\nu(k))\implies(\mu^{\nu}(i))(x)=\mu(\varphi_{i}(k)).
\]
\end{defn}

The image of $\mu^{\nu}$ is precisely the set of $(\nu,\mu)$-computable
functions. 

The numbering associated to computable sequences of point of $X$
is thus $\nu^{\text{id}_{\mathbb{N}}}$.

\subsection{Semi-decidable sets and co-semi-decidable sets }

Let $\nu$ be a multi-numbering of $X$. A subset $A$ of $X$ is
$\nu$\emph{-semi-decidable} when there is an algorithm that stops
exactly on the $\nu$-names of elements of $A$. 

We associate a numbering denoted $\nu_{SD}$ to semi-decidable subsets
of $X$. 

It is defined as follows:
\begin{defn}
Let $\nu$ be a multi-numbering on $X$. A subset $A$ of $X$ is
\emph{$\nu$-semi-decidable} if and only if there is a recursive function
$\varphi_{n}$ such that 
\[
\forall i\in\text{dom}(\nu),\,\varphi_{n}(i)\downarrow\iff\nu(i)\subseteq A.
\]

We define a subnumbering of $\mathcal{P}(X)$ by putting $\nu_{SD}(n)=A$
for each $n$ and $A$ as above. 
\end{defn}

One extends the previous definitions to numberings by using the identification
$\nu\mapsto\hat{\nu}$ between numberings and multi-numberings. 

A set is\emph{ $\nu$-co-semi-decidable }if its complement is $\nu$-semi-decidable.
We define as above a subnumbering $\nu_{coSD}$ of $\mathcal{P}(X)$.
For a subset $A$ of $X$, we denote by $A^{c}$ the complement of
$A$. Put: 
\[
\text{dom}(\nu_{coSD})=\text{dom}(\nu_{SD}),
\]
\[
\forall n\in\text{dom}(\nu_{SD}),\,\nu_{coSD}(n)=\nu_{SD}(n)^{c}.
\]

\subsection{Intrinsic embeddings\label{subsec:Intrinsic-embeddings}}

Let $(X,\nu)$ be a numbered set, and $A\subseteq X$ a subset of
$X$. The \emph{restriction of $\nu$ to $A$} is the numbering $\nu_{\vert A}$
of $A$ defined as follows: 
\[
\text{dom}(\nu_{\vert A})=\{n\in\text{dom}(\nu),\,\nu(n)\in A\},
\]
\[
\forall n\in\text{dom}(\nu_{\vert A}),\,\nu_{\vert A}(n)=\nu(n).
\]

If $B\subseteq A\subseteq X$, we say that $B$ is \emph{a $\nu$-semi-decidable
subset of $A$ }if $B$ is $\nu_{\vert A}$-semi-decidable. 

When $A$ is a subset of the numbered set $(X,\nu)$, we define a
subnumbering $(\nu_{SD})_{\cap A}$ of $\mathcal{P}(A)$ as follows:
\[
\text{dom}((\nu_{SD})_{\cap A})=\text{dom}(\nu_{SD});
\]
\[
\forall n\in\text{dom}((\nu_{SD})_{\cap A}),\,(\nu_{SD})_{\cap A}(n)=\nu_{SD}(n)\cap A.
\]

Thus $(\nu_{SD})_{\cap A}$ is simply the numbering of $\nu$-semi-decidable
sets (these are subsets of $X$), which we see as a subnumbering of
the set of subsets of $A$ by intersecting $\nu$-semi-decidable sets
with $A$. It is clear that such sets are always $\nu_{\vert A}$-semi-decidable: 
\begin{prop}
\label{prop: Semi-decidable set intersected }We have 
\[
(\nu_{\vert A})_{SD}\ge(\nu_{SD})_{\cap A}.
\]
\end{prop}

\begin{defn}
[Bauer, \cite{Bauer2023}] Let $(X,\nu)$ be a numbered set. A subset
$A$ of $X$ is \emph{an intrinsic subset} of $(X,\nu)$ (or is \emph{intrinsically
embedded} in $(X,\nu)$) if the following equivalence of subnumberings
holds:
\[
(\nu_{\vert A})_{SD}\equiv(\nu_{SD})_{\cap A}.
\]
\end{defn}

Note that in particular these subnumbering should have the same image.
\begin{example}
[\cite{Bauer2023}] The following remarkable example of a non-intrinsic
embedding is based on Friedberg's example of a semi-decidable subset
of the Cantor space which is not open. A computable sequence of natural
numbers can be seen as a semi-decidable subset of $\mathbb{N}^{2}$,
by identifying the sequence $(u_{n})_{n\in\mathbb{N}}$ with its graph
$\{(n,u_{n}),n\in\mathbb{N}\}$. This defines a computable embedding,
but it is not an intrinsic embedding. See \cite{Bauer2023} for more
details. 
\end{example}

\subsection{Diagrams of relative semi-decidability }

If $(X,\nu)$ is a numbered set, we define a relation on $\mathcal{P}(X)$,
denoted $\rightsquigarrow_{\nu}$, defined by
\[
A\rightsquigarrow_{\nu}B\iff B\text{ is not a \ensuremath{\nu}-semi-decidable subset of }A\cup B.
\]
In other words $A\rightsquigarrow_{\nu}B$ holds if and only if $B\text{ is not }\nu_{\vert A\cup B}\text{ semi-decidable}$.
We call this relation the \emph{diagram of relative semi-decidability
of} $(X,\nu)$. It is in general not a transitive relation. 

In a topological space $X$ with a subset $A$, the points that cannot
``topologically'' be told apart from those of $A$ are the points
in the set $\overline{A}$, the closure of $A$. Thus we can expect
that, in contexts where computable functions are automatically continuous,
the relation $\rightsquigarrow_{\nu}$ should be related to the relation
$R$ given by: 
\[
R(A,B)\iff B\cap\overline{A}\ne\emptyset.
\]
The Markov conditions introduced in this article are precisely used
to formalize a link between $R$ and $\rightsquigarrow_{\nu}$.
\begin{example}
Let $W$ be the subnumbering of $\mathcal{P}(\mathbb{N})$ defined
by $W_{n}=\text{dom}(\varphi_{n})$, i.e. the standard numbering of
the set ${\displaystyle \Sigma_{1}^{0}}$ of c.e. subsets of $\mathbb{N}$.
We have, for any two r.e. subsets $A$ and $B$ of $\mathbb{N}$:
\[
\{A\}\rightsquigarrow_{W}\{B\}\iff B\subsetneq A.
\]
This follows easily from the Rice-Shapiro Theorem, which characterizes
$W$-semi-decidable sets, see \cite[P. 130]{Cutland1997}. The Rice-Shapiro
theorem thus characterizes those sets of c.e. sets $\mathcal{X}$
for which $\mathcal{X}\rightsquigarrow_{W}{\displaystyle \Sigma_{1}^{0}}$
holds. But this is very far from giving a complete description of
$\rightsquigarrow_{W}$. 
\begin{problem}
Characterize $\rightsquigarrow_{W}$. 
\end{problem}

\end{example}

\subsection{Computable topological spaces }

We follow the general definition given in \cite{RAUZY2023}. Discussion
of the differences with other definitions can be found there. Note
that Bauer introduced the exact same definition in \cite{Bauer2023}
in the context of synthetic topology, precisely in order to prove
the continuity theorem of Spreen in synthetic topology. 
\begin{defn}
\label{def:MAIN DEF}A \emph{Type 1 computable topological space}
is a quadruple $(X,\nu,\mathcal{T},\tau)$ where $X$ is a set, $\mathcal{T}\subseteq\mathcal{P}(X)$
is a topology on $X$, $\nu:\subseteq\mathbb{N}\rightarrow X$ is
a numbering of $X$, $\tau:\subseteq\mathbb{N}\rightarrow\mathcal{T}$
is a subnumbering of $\mathcal{T}$, and such that: 
\begin{enumerate}
\item The image of $\tau$ generates the topology $\mathcal{T}$;
\item The empty set and $X$ both belong to the image of $\tau$;
\item (Malcev's Condition). The open sets in the image of $\tau$ are uniformly
semi-decidable, i.e. $\tau\le\nu_{SD}$; 
\item The operations of taking (computable) unions and finite intersections
are computable, i.e.:
\begin{enumerate}
\item The function $\bigcup:\mathcal{T}^{\mathbb{N}}\rightarrow\mathcal{T}$
which maps a sequence of open sets $(A_{n})_{n\in\mathbb{N}}$ to
its union is $(\tau^{\text{id}_{\mathbb{N}}},\tau)$-computable;
\item The function $\bigcap:\mathcal{T}\times\mathcal{T}\rightarrow\mathcal{T}$
which maps a pair of open sets to its intersection is $(\tau\times\tau,\tau)$-computable. 
\end{enumerate}
\end{enumerate}
\end{defn}

The sets in the image of $\tau$ are called the \emph{effective open}
sets. 

This is in fact not the most general definition one should consider,
as we could allow $\tau$ and $\nu$ to be multi-numberings, leaving
all other conditions as they are. But we will not use this more general
definition here, and so we only mention it. 
\begin{defn}
A function $f:(X,\nu,\mathcal{T}_{1},\tau_{1})\rightarrow(Y,\mu,\mathcal{T}_{2},\tau_{2})$
is \emph{effectively continuous }if it is (classically) continuous,
and if furthermore the function $f^{-1}:\mathcal{T}_{2}\rightarrow\mathcal{T}_{1}$
is $(\tau_{2},\tau_{1})$-computable. 
\end{defn}

Note that if $f^{-1}:\mathcal{T}_{2}\rightarrow\mathcal{T}_{1}$ is
$(\tau_{2},\tau_{1})$-computable, it should map open sets that belong
to the image of $\tau_{2}$ to open sets that belong to the image
of $\tau_{1}$. 

Note also that a function can be effectively continuous without being
computable, for instance if $(\mathcal{T}_{2},\tau_{2})$ is the indiscrete
topology, every function to $(Y,\mu,\mathcal{T}_{2},\tau_{2})$ is
effectively continuous, without having to be computable. Such extreme
examples do not occur when we consider computable topologies close
to the Ershov topology, and in particular for computably sober spaces,
see the Section \ref{subsec:Sober-spaces}. 

\subsection{Examples of Type 1 computable topological spaces }

\subsubsection{Ershov topology }

Any set $X$ equipped with a numbering $\nu$ can be turned into a
computable topological space thanks to the Ershov topology: denote
by $\mathcal{T}_{SD}$ the topology generated by semi-decidable sets,
and by $\nu_{SD}$ the subnumbering of $\mathcal{T}_{SD}$ associated
to semi-decidable sets. The following proposition is straightforward. 
\begin{prop}
For any numbered set $(X,\nu)$, $(X,\nu,\mathcal{T}_{SD},\nu_{SD})$
is a Type 1 computable topological space. 
\end{prop}

The Ershov topology can be seen as the topology that is closest to
the discrete topology amongst topologies that are ``constructively
acceptable'' \cite{Bauer2023}. But it does not have to be discrete.
It shares the following (obvious) characterization of the classical
discrete topology: 
\begin{prop}
Any computable function $f$ from $(X,\nu,\mathcal{T}_{SD},\nu_{SD})$
to a computable topological space $(Y,\mu,\mathcal{T}_{2},\tau_{2})$
is effectively continuous. 
\begin{prop}
The property above characterizes the Ershov topology. 
\end{prop}

\end{prop}

These two propositions show that there are two distinct kinds of continuity
results one often encounters: 
\begin{itemize}
\item Results that state that computable functions defined on a space $(X,\nu,\mathcal{T}_{1},\tau_{1})$
are always effectively continuous, with no restriction on their codomain
(other than it being a computable topological space). These are in
fact results that characterize the semi-decidable sets on a given
numbered set, showing that $(\mathcal{T}_{1},\tau_{1})=(\mathcal{T}_{SD},\tau_{SD})$
(more precisely that $\tau_{1}\equiv\tau_{SD}$ for the equivalence
of numberings). They can thus be expressed and proved inside the numbered
set $(X,\nu)$, with no reference to computable functions mapping
to other computable topological spaces. Among these are the Rice-Shapiro
Theorem and continuity results for Scott domains (see for instance
\cite{Spr98}). 
\item Results that have limitations on the codomain. For instance the theorems
of Kreisel-Lacombe-Schoenfield \cite{Kreisel1957}, Ceitin \cite{Ceitin1967}
and Moschovakis \cite{Moschovakis1964} are expressed only for function
whose codomain is a recursive metric space, and make actual use of
properties of the codomain. 
\end{itemize}
We are interested here in results of the second kind. 

\subsubsection{The computable reals }

We first define a numbering $c_{\mathbb{Q}}$ of the set of rationals.
The numbering $c_{\mathbb{Q}}$ is defined on $\mathbb{N}$. Given
a natural number $n$, seen as encoding a triple $n=\langle a,b,c\rangle$
via a pairing function, we put $c_{\mathbb{Q}}(n)=(-1)^{a}\frac{b}{c+1}$. 

We now define the Cauchy subnumbering $c_{\mathbb{R}}$ of $\mathbb{R}$.
Denote by $(\varphi_{0},\varphi_{1},\varphi_{2}...)$ a standard enumeration
of all partial recursive functions. 
\begin{defn}
The Cauchy subnumbering of $\mathbb{R}$ is defined as follows:
\[
\text{dom}(c_{\mathbb{R}})=\{i\in\mathbb{N},\,\exists x\in\mathbb{R},\,\forall n\in\mathbb{N},\,\left|c_{\mathbb{Q}}(\varphi_{i}(n))-x\right|<2^{-n}\};
\]
\[
\forall i\in\text{dom}(c_{\mathbb{R}}),\,c_{\mathbb{R}}(i)=\lim_{n\rightarrow\infty}(c_{\mathbb{Q}}(\varphi_{i}(n))).
\]

The real numbers in the image of $c_{\mathbb{R}}$ are called the
\emph{computable reals}. Let $\mathbb{R}_{c}$ be the set of computable
reals. 
\end{defn}

Now we equip $(\mathbb{R}_{c},c_{\mathbb{R}})$ with a computable
topology. 

Let $\mathcal{T}_{\mathbb{R}}$ be the Euclidian topology on $\mathbb{R}_{c}$. 

For $p,q\in\mathbb{R}_{c}$, we denote by $]p,q[_{c}=\{x\in\mathbb{R}_{c},p<x<q\}$.
For convenience, we do not require $p<q$, we simply have $]p,q[_{c}=\emptyset$
if this is not the case. 

As usual, $W_{i}=\text{dom}(\varphi_{i})$. Define a subnumbering
$\tau_{\mathbb{R}}$ of $\mathcal{T}_{\mathbb{R}}$ as follows: 
\[
\text{dom}(\tau_{\mathbb{R}})=\mathbb{N};
\]
\[
\forall i\in\mathbb{N},\,\tau_{\mathbb{R}}(i)=\bigcup_{\langle n,m\rangle\in W_{i}}]c_{\mathbb{Q}}(n),c_{\mathbb{Q}}(m)[.
\]

We do not check the following easy lemma.
\begin{lem}
The quadruple $(\mathbb{R}_{c},c_{\mathbb{R}},\mathcal{T}_{\mathbb{R}},\tau_{\mathbb{R}})$
is a Type 1 computable topological space. 
\end{lem}

An important theorem of Friedberg states: 
\begin{thm}
[Friedberg, \cite{Friedberg1958}]The pair $(\mathcal{T}_{\mathbb{R}},\tau_{\mathbb{R}})$
is not the Ershov topology on $(\mathbb{R}_{c},c_{\mathbb{R}})$:
there exists a $c_{\mathbb{R}}$-semi-decidable subset of $\mathbb{R}_{c}$
which is not open in $\mathcal{T}_{\mathbb{R}}$. 
\end{thm}

\subsubsection{\label{subsec:A-converging-sequence-and-its-limit}A converging sequence
and its limit}

Let $\mathbb{N}^{+}$ be the set of non-decreasing sequences of $\{0,1\}^{\mathbb{N}}$. 

Let $(\varphi_{0},\varphi_{1},\varphi_{2},...)$ be a standard enumeration
of all partial recursive function. We define a numbering $c_{\mathbb{N}^{+}}$
of $\mathbb{N}^{+}$ by the following:
\[
\text{dom}(c_{\mathbb{N}^{+}})=\{i\in\mathbb{N},\,\varphi_{i}\in\{0,1\}^{\mathbb{N}},\,\forall n\in\mathbb{N},\,\varphi_{i}(n+1)\ge\varphi_{i}(n)\}
\]
\[
\forall i\in\text{dom}(c_{\mathbb{N}^{+}}),\,c_{\mathbb{N}^{+}}(i)=\varphi_{i}.
\]

Note that we can identify $\mathbb{N}^{+}$ with $\mathbb{N}\cup\{\infty\}$
via the identification 
\[
0^{n}1^{\infty}\leftrightarrow n;
\]
\[
0^{\infty}\leftrightarrow\infty.
\]
We then naturally extend the linear order of $\mathbb{N}$ to $\mathbb{N}^{+}$. 

We define a computable topology on $(\mathbb{N}^{+},c_{\mathbb{N}^{+}})$
as follows. Consider the basis: 
\[
\mathcal{B}=\{\{n\},n\in\mathbb{N}\}\cup\{\{k\in\mathbb{N}^{+},k\ge n\},\,n\in\mathbb{N}\},
\]
equipped with the numbering $\beta$ defined by $\beta(2n)=\{n\}$
and $\beta(2n+1)=\{k\in\mathbb{N}^{+},k\ge n\}$ for any $n\in\mathbb{N}$. 

Let $\mathcal{T}_{\mathbb{N}^{+}}$ be the classical topology generated
by $\mathcal{B}$. Let $W_{i}=\text{dom}(\varphi_{i})$ be the standard
numbering of c.e. subsets of $\mathbb{N}$. We then define a subnumbering
$\tau_{\mathbb{N}^{+}}$ of $\mathcal{T}_{\mathbb{N}^{+}}$ by:
\[
\text{dom}(\tau_{\mathbb{N}^{+}})=\mathbb{N},
\]
\[
\forall i\in\mathbb{N},\,\tau_{\mathbb{N}^{+}}(i)=\bigcup_{n\in W_{i}}\beta(n).
\]

The following fact is then easy to check:
\begin{lem}
The quadruple $(\mathbb{N}^{+},c_{\mathbb{N}^{+}},\mathcal{T}_{\mathbb{N}^{+}},\tau_{\mathbb{N}^{+}})$
is a Type 1 computable topological space. 
\end{lem}

Note that even in this simple case the Ershov topology of $(\mathbb{N}^{+},c_{\mathbb{N}^{+}})$
is different from $(\mathcal{T}_{\mathbb{N}^{+}},\tau_{\mathbb{N}^{+}})$
(and even: there is a semi-decidable set which does not belong to
$\mathcal{T}_{\mathbb{N}^{+}}$), see \cite{Hoyrup2016}. 

\subsection{\label{subsec:Sober-spaces}Sober spaces }

The definition of a \emph{computably sober space} can be traced back
to Paul Taylor in \cite{Taylor2002}, and to Mathias Schröder \cite{Schroeder2003},
who, in the context of Type 2 computable analysis, introduced it under
the name of ``computably admissible representation''. Here, we chose
to follow the vocabulary of Paul Taylor, which is now commonly used
in synthetic topology, see \cite{Bauer2023}. (The main reason for
this choice is that the expression ``admissible numbering'' already
has several meanings unrelated to the present investigation.)

Let $(X,\nu,\mathcal{T},\tau)$ be a Type 1 computable topological
space. 

We use the subnumbering $\tau$ to define a multi-numbering $\tau^{*}$
of $X$, by the following: 
\[
x\in\tau^{*}(i)\iff(\forall n\in\text{dom}(\tau),\,\varphi_{i}(n)\downarrow\iff x\in\tau(n))
\]
In other words, the $\tau^{*}$-name of a point is a program that
stops exactly on the $\tau$-names of open sets that contain $x$. 
\begin{prop}
In any Type 1 computable topological space $(X,\nu,\mathcal{T},\tau)$,
we have $\nu\le\tau^{*}$. 
\end{prop}

\begin{proof}
This follows from the fact that sets in the image of $\tau$ are uniformly
$\nu$-semi-decidable: there is an algorithm that, given the $\nu$-name
of a point and the $\tau$-name of a set, strops if and only if the
point belongs to the set. 
\end{proof}
In particular, it is always the case that every element of $X$ has
a $\tau^{*}$-name, and $\tau^{*}$ is indeed a multi-numbering, and
not only a multi-subnumbering. 
\begin{defn}
A Type 1 computable topological space $(X,\nu,\mathcal{T},\tau)$
is \emph{computably sober} if $\tau^{*}\equiv\nu$. 
\end{defn}

Note that this equivalence in particular implies that $\tau^{*}$
should be an actual numbering, and not a multi-numbering. 
\begin{example}
The quadruple $(\mathbb{R}_{c},c_{\mathbb{R}},\mathcal{T}_{\mathbb{R}},\tau_{\mathbb{R}})$
is computably sober, as can easily be checked. This shows that a space
can be computably sober without being equipped with the Ershov topology. 
\end{example}

The following results are well known in the context of Type 2 computable
analysis \cite{Schroeder2003}. 
\begin{prop}
Let $f:(X,\nu,\mathcal{T}_{1},\tau_{1})\rightarrow(Y,\mu,\mathcal{T}_{2},\tau_{2})$
be an effectively continuous function between Type 1 computable topological
spaces. Suppose that $(Y,\mu,\mathcal{T}_{2},\tau_{2})$ is computably
sober. Then $f$ is computable. 
\end{prop}

\begin{proof}
Given the $\nu$-name of a point $x$, we compute a $\mu$-name of
$f(x)$. Because $(Y,\mu,\mathcal{T}_{2},\tau_{2})$ is computably
sober, it suffices to produce a program that, given the $\tau_{2}$-name
of an open set, stops if and only if $f(x)$ belongs to this set.
Such a program is as follows: given the $\tau_{2}$-name of $O$,
compute the $\tau_{1}$-name of $f^{-1}(O)$ (since $f$ is effectively
continuous), and then check whether $x\in f^{-1}(O)$. 
\end{proof}
\begin{prop}
Let $(Y,\mu,\mathcal{T}_{2},\tau_{2})$ be a Type 1 computable topological
space. Suppose that for every Type 1 computable topological space
$(X,\nu,\mathcal{T}_{1},\tau_{1})$ and every effectively continuous
function $f:(X,\nu,\mathcal{T}_{1},\tau_{1})\rightarrow(Y,\mu,\mathcal{T}_{2},\tau_{2})$,
$f$ is computable. 

Then $(Y,\mu,\mathcal{T}_{2},\tau_{2})$ is computably sober. 
\end{prop}

\begin{proof}
We can equip $Y$ with the multi-numbering $\tau_{2}^{*}$.

We first show that $\tau_{2}^{*}$ must be a numbering. 

Suppose that $x$ and $y$ are two points of $Y$ that belong to exactly
the same effectively open sets. Let $\nu$ be any numbering of $\{0,1\}$.
Let $\mathcal{I}=\{\emptyset,\{0,1\}\}$ be the indiscrete topology
on $\{0,1\}$, equipped with the numbering $i$ defined by $i(p)=\emptyset$
if $\varphi_{p}(p)\uparrow$ and $i(p)=\{0,1\}$ otherwise. We leave
it to the reader to check that computable unions and finite intersections
are computable for $i$. 

Note that whatever $\nu$, $(\{0,1\},\nu,\mathcal{I},i)$ is a Type
1 computable topological space, because the empty set and $\{0,1\}$
are $\nu$-semi-decidable. 

The map $g:(\{0,1\},\nu,\mathcal{I},i)\rightarrow(Y,\mu,\mathcal{T}_{2},\tau_{2})$
defined by $g(0)=x$ and $g(1)=y$ is thus effectively continuous.
And thus it should be computable. 

But then, $g$ is a reduction from the equality problem in $(\{0,1\},\nu)$
to the equality problem in $\mu$. But since $\nu$ is an arbitrary
numbering of $\{0,1\}$, the equality problem for $\nu$ can be arbitrarily
hard, and in particular it can be strictly harder than that of $(Y,\mu)$.
Thus $\tau_{2}^{*}$ must be a numbering. 

Now that we know that $\tau_{2}^{*}$ is a numbering, it is immediate
to see that $(Y,\tau_{2}^{*},\mathcal{T}_{2},\tau_{2})$ is a Type
1 computable topological space. The identity $\text{id}_{Y}:(Y,\tau_{2}^{*},\mathcal{T}_{2},\tau_{2})\rightarrow(Y,\nu,\mathcal{T}_{2},\tau_{2})$
is effectively continuous, trivially. Thus it is computable, and thus
$\tau_{2}^{*}\le\nu$.
\end{proof}

\subsection{Recursive metric spaces }

The following definition is that of Moschovakis in \cite{Moschovakis1964}. 
\begin{defn}
A \emph{recursive metric space} (RMS) is a triple $(X,\nu,d)$, where
$(X,d)$ is a countable metric space, $\nu$ is a numbering of $X$,
such that $d$ is a $(\nu\times\nu,c_{\mathbb{R}})$-computable function. 
\begin{defn}
[\cite{Spr98},\cite{Kushner1984}]\label{def:Algo passage limit }
A RMS $(X,\nu,d)$\emph{ admits an algorithm of passage to the limit}
if there is a program that on input the $\nu^{\text{id}_{\mathbb{N}}}$-name
of a computable sequence that converges to a computable point and
the code for a function which bounds the convergence speed of this
sequence produces a $\nu$-name for its limit. 
\end{defn}

\end{defn}

It is well known that the above condition is equivalent to asking
that there exist an algorithm that on input a sequence that converges
faster than $n\mapsto2^{-n}$ produces its limit. 

This condition is called ``condition (A)'' in \cite{Moschovakis1964}. 

Note that defining the computable topology of a recursive metric space
is less trivial than it might appear at first glance. In \cite{RAUZY2023},
basing ourselves on the work of Dieter Spreen \cite{Spr98}, we give
the following definition, which uses the numbering $c_{\nearrow}$
of \emph{left computable reals }(i.e. the $c_{\nearrow}$-name of
a real $x$ is the code of a Turing Machine that produces an increasing
sequence of rationals that converges to $x$):
\begin{defn}
Let $(X,\nu,d)$ be a RMS. Let $\mathcal{T}$ be the classical metric
topology on $X$. We define a subnumbering $\tau$ of $\mathcal{T}$
as follows: 
\[
\tau(\langle n,m\rangle)=O
\]
 if and only if:
\begin{itemize}
\item $O$ is a $\nu$-semi-decidable set, and $n$ is a $\nu_{SD}$-name
of $O$;
\item $m$ is the code of a $(\nu,c_{\nearrow})$-computable function $F$
which satisfies the following conditions:
\[
\forall x\in O,\,B(x,F(x))\subseteq O;
\]
\[
\forall x\in O,\,\exists r>0,\,\forall y\in B(x,r),\,F(y)>r.
\]
\end{itemize}
See \cite{RAUZY2023} for more details. Note that we prove in \cite{RAUZY2023}
that when $(X,\nu,d)$ is computably separable, the above definition
amounts to requiring that open sets in the image of $\tau$ be computable
unions of open balls. 
\end{defn}

\section{\label{part: On the continuity of computable functions}Definitions
and results}

\subsection{\label{subsec:notions-of-Eff-Closure}Different notions of effective
closure }

Let $(X,\nu,\mathcal{T},\tau)$ be a computable topological space. 

We introduce here three notions of effective closure. 
\begin{defn}
If $A$ is a subset of $X$, say that $x$ is \emph{effectively adherent}
to $A$ if there is a program that, on input the $\tau$-name of an
open set $O$ that contains $x$, produces the $\nu$-name of a point
$y$ in $O\cap A$. 

The \emph{effective closure} of $A$ is the set of all points that
are effectively adherent to $A$. It is denoted $\overline{A}^{+}$.
\end{defn}

Associated to this definition is a multi-numbering $\mu_{\overline{A}^{+}}:\subseteq\mathbb{N}\rightrightarrows\overline{A}^{+}$,
defined as follows: the name of a point $x$ for $\mu_{\overline{A}^{+}}$
is the code for the algorithm that, given the $\tau$-name of an open
set which contains $x$, produces the $\nu$-name of a point $Y$
in $O\cap A$. (This multi-numbering does not have to be a numbering,
even when $X$ satisfies separation axioms, like being Hausdorff.)

Remark that when the computable topology $(\mathcal{T},\tau)$ was
defined thanks to a basis, one can work with basic open sets instead
of open sets in the definition above. 

And we define also the sequential closure.
\begin{defn}
If $A$ is a subset of $X$, the \emph{effective sequential closure
}of $A$ is the set of all limits of converging computable sequences
of elements of $A$. 

We denote it by $\overline{A}^{+\text{seq}}$. 
\end{defn}

Associated to this definition is a multi-numbering $\mu_{\overline{A}^{+\text{seq}}}$,
a point of $\overline{A}^{+\text{seq}}$ is described by the code
of a $\nu$-computable sequence that converges to it. (This multi-numbering
is a numbering as soon as $X$ is Hausdorff.)

Finally, we defined the normed effective sequential closure. 

Let $(u_{n})_{n\in\mathbb{N}}$ be a computable sequence in $(X,\nu,\mathcal{T},\tau)$
that converges to a point $x$. We say that $(u_{n})_{n\in\mathbb{N}}$
is \emph{computably normed} if there is a program that, on input the
$\tau$-name of an open set $O$ that contains $x$, will produce
$N\in\mathbb{N}$ such that 
\[
\forall n\ge N,\,u_{n}\in O.
\]

\begin{defn}
If $A$ is a subset of $X$, the \emph{normed sequential closure }of
$A$ is the set of all limits of converging computably normed computable
sequences of elements of $A$. We denote it by $\overline{A}^{+\text{seq},N}$. 
\end{defn}

Associated to this definition is a multi-numbering $\mu_{\overline{A}^{+\text{seq,}N}}$,
a point of $\overline{A}^{+\text{seq},N}$ is described by the code
of a $\nu$-computable sequence that converges to it, together with
the code of an algorithm that shows that this sequence is normed.
(This multi-numbering is a numbering as soon as $X$ is Hausdorff.)

The following proposition is crucial in understanding the use of normed
sequences. 
\begin{prop}
\label{prop:Comp Seq Normed gives map to Sober space }A computable
sequence $(u_{n})_{n\in\mathbb{N}}$ converging to a point $x$ in
$(X,\nu,\mathcal{T},\tau)$ is computably normed if and only if the
map $f:(\mathbb{N}^{+},c_{\mathbb{N}^{+}},\mathcal{T}_{\mathbb{N}^{+}},\tau_{\mathbb{N}^{+}})\rightarrow(X,\nu,\mathcal{T},\tau)$
given by $f(n)=u_{n}$ and $f(\infty)=x$ is effectively continuous. 
\end{prop}

\begin{proof}
Suppose that $f$ is effectively continuous. 

Given the $\tau$-name of a set $O$ that contains $x$, compute a
$\tau_{\mathbb{N}^{+}}$-name of its preimage in $\mathbb{N}^{+}$.
By definition of the numbering $\tau_{\mathbb{N}^{+}}$, this open
set can be written as a computable union of basic sets which are either
singletons in $\mathbb{N}$, or sets of the form $\{k\in\mathbb{N}^{+},\,k\ge N\}$,
for $N\in\mathbb{N}$. Since $f^{-1}(O)$ contains $\infty$, it must
contain a basic set of the form $\{k\in\mathbb{N}^{+},\,k\ge N\}$,
and such a set can be found by exhaustive search. The computed $N$
will precisely be an index for which $\forall n\ge N,\,u_{n}\in O$.
This shows that in the sequence $(u_{n})$ is normed. 

Suppose now that $(u_{n})_{n\in\mathbb{N}}$ is normed. Let $O$ be
an open set of $X$ given by a $\tau$-name. 

The preimage $f^{-1}(O)\cap\mathbb{N}$ is a semi-decidable subset
of $\mathbb{N}$, which can thus be written as a computable union
of singletons, thus giving a $\tau_{\mathbb{N}^{+}}$-name of it.
Because $(u_{n})_{n\in\mathbb{N}}$ is normed, we can compute $N$
such that $\forall n\ge N,\,u_{n}\in O$. This implies that the basic
set $\{k\in\mathbb{N}^{+},\,k\ge N\}$ is a subset of $f^{-1}(O)$.
But note that 
\[
f^{-1}(O)=(f^{-1}(O)\cap\mathbb{N})\cup\{k\in\mathbb{N}^{+},\,k\ge N\}.
\]
Thus, as it is possible to compute $\tau_{\mathbb{N}^{+}}$-names
for both sets of the right hand side in the above, it is also possible
to obtain the $\tau_{\mathbb{N}^{+}}$-name of their union.
\end{proof}
\begin{rem}
In a recursive metric space, the normed sequential closure can be
defined in terms of the metric: $\overline{A}^{+\text{seq},N}$ is
the set of all limits of computable sequences of points of $A$ that
converge at a computable speed. 

The notion of normed sequence that we propose here is only one of
several possible notions that generalize sequences in metric spaces
that converge computably fast. 

In particular, a distinct generalization was considered by Spreen
in \cite{Spr98} using bases that admit a strong inclusion relation.
See Definition 2.7 and Definition 2.10 of \cite{Spr98}, and Lemma
3.2 where it is shown that the notion, which is defined for general
topological spaces, is indeed an extension of the notion for metric
spaces. 

We could introduce notions of effective discontinuities, and identification
conditions, and so on, with respect to this notion of effective closure.
However, we leave this to the reader, since these are really syntactic
changes that would not involve any new reasonings. 
\end{rem}

Note that Type 1 computable topological spaces are always countably
based, and thus their topology is determined by converging sequences.
However, not every point needs to have a computably enumerable basis
of neighborhood (see Section \ref{subsec:Multi-numbering-of-neighborhood bases }).
Because of this, the effective closure does not have to correspond
to the effective sequential closure, and thus the computable topology
does not have to be determined by computable sequences that converge. 

The following easy propositions are left to the reader: 
\begin{prop}
The following inclusions always hold:
\[
A\subseteq\overline{A}^{+\text{seq,}N}\subseteq\overline{A}^{+\text{seq}}\subseteq\overline{A}^{+}\subseteq\overline{A}.
\]
\begin{prop}
Each inclusion given in the above proposition can be strict. 
\end{prop}

\end{prop}

\subsection{Effective discontinuities }

We define different notions of effective discontinuities thanks to
our different notions of effective closure. 

Classically, a discontinuity at a point $x$ for a function $f$ is
defined by a neighborhood of $f(x)$ whose preimage is not a neighborhood
of $x$. To obtain an effective statement, we use a positive definition,
equivalent to the previous one: a discontinuity at a point $x$ for
a function $f$ is defined by a neighborhood $B$ of $f(x)$, such
that $x$ is adherent to the complement of the preimage of $B$. 
\begin{defn}
Let $f$ be a map between computable topological spaces $(X,\nu,\mathcal{T}_{1},\tau_{1})$
and $(Y,\mu,\mathcal{T}_{2},\tau_{2})$. \emph{An effective discontinuity
for $f$ at $x$ }is given by an effective open set $O_{2}\in\mathcal{T}_{2}$,
which contains $f(x)$, but such that $x$ belongs to the effective
closure of the complement of the preimage of $O_{2}$. 

Summarized, this gives: $f$ is effectively discontinuous at $x$
if and only if:
\[
\exists O_{2}\in\mathcal{T}_{2},\,x\in f^{-1}(O_{2})\cap\overline{(f^{-1}(O_{2}))^{\boldsymbol{c}}}^{+}.
\]
\end{defn}

There is naturally a numbering associated to effective discontinuities:
a triple ($\nu$-name of $x$, $\tau_{2}$-name of $O_{2}$, $\mu_{\overline{(f^{-1}(O_{2}))^{\boldsymbol{c}}}^{+}}$-name
of $x$ (i.e. code for the program that shows that $x$ is effectively
adherent to $f^{-1}(O_{2}){}^{\boldsymbol{c}}$)) encodes all the
required data about the effective discontinuity. 

We have a similar definitions for the sequential closure: 
\begin{defn}
Let $f$ be a map between computable topological spaces $(X,\nu,\mathcal{T}_{1},\tau_{1})$
and $(Y,\mu,\mathcal{T}_{2},\tau_{2})$.\emph{ }Say that $f$ is \emph{effectively
sequentially discontinuous at $x$ }if and only if:
\[
\exists O_{2}\in\mathcal{T}_{2},\,x\in f^{-1}(O_{2})\cap\overline{(f^{-1}(O_{2}))^{\boldsymbol{c}}}^{+\text{seq}}.
\]
\end{defn}

In other words, there should be a computable sequence of points that
converges to $x$, but whose images by $f$ never meet a certain neighborhood
of $f(x)$. There is also a numbering associated to sequential effective
discontinuities: such a discontinuity is encoded by a triple consisting
of a $\nu$-name of $x$, a $\tau_{2}$-name of $O_{2}$, and the
code for a $\nu$-computable sequence that converges to $x$ while
its image by $f$ avoids $O_{2}$. 

And finally for the normed sequential closure. 
\begin{defn}
Let $f$ be a map between computable topological spaces $(X,\nu,\mathcal{T}_{1},\tau_{1})$
and $(Y,\mu,\mathcal{T}_{2},\tau_{2})$.\emph{ }Say that $f$ has
a \emph{normed effective sequential discontinuity at $x$ }if and
only if:
\[
\exists O_{2}\in\mathcal{T}_{2},\,x\in f^{-1}(O_{2})\cap\overline{(f^{-1}(O_{2}))^{\boldsymbol{c}}}^{+\text{seq,N}}.
\]
\end{defn}

We leave it to the reader to define the associated numbering. 
\begin{rem}
In \cite{BRATTKA2023}, Vasco Brattka introduces a notion of \emph{computably
discontinuous} \emph{problem} in the Weihrauch lattice. This notion
is often similar to the notions of effective discontinuities that
we consider here. See Example 18 of \cite{BRATTKA2023}, where a family
of functions $f_{A}$ are defined, $A\subseteq\mathbb{N}$. These
are defined on countable sets, and $f_{A}$ is computably discontinuous
in the sense of \cite{BRATTKA2023} exactly when $f_{A}$ has a computable
discontinuity in the sense of the present article. 

The notion of computably discontinuous problem of \cite{BRATTKA2023}
is obtained as a Type 2 generalization of the notion of \emph{productive
set}. If we translate this notion back to the Type 1 context, we obtain
a notion of productivity for problems, and not simply for sets. Recall
that a \emph{Type 1 problem} is a multi-function between numbered
sets. 
\begin{defn}
A Type 1 problem $P:(X,\nu)\rightrightarrows(Y,\mu)$ is called \emph{productive}
if there is a computable function $f$ which, given $i$, produces
$n=f(i)$ such that either: 
\end{defn}

\begin{itemize}
\item $n$ does not belong to $\text{dom}(\varphi_{i})$ while $n$ belongs
to $\text{dom}(\nu)$; 
\item Or $n$ belongs to $\text{dom}(\varphi_{i})\cap\text{dom}(\nu)$ and
$\mu(\varphi_{i}(n))\notin P(\nu(n)).$
\end{itemize}
One should probably explore in more details the relation between the
notion of computable discontinuity at a point considered here and
the notion of computably discontinuous problem introduced in \cite{BRATTKA2023}.
We however leave this for future research. (In particular, one important
aspect of the definition of a computably discontinuous problem is
that it can be applied to multi-valued functions, contrary to the
definitions we present here. Note that Spreen studied continuity of
Type 1 computable multi-functions in \cite{Spreen_2010}. Is it possible
to obtain results similar to the ones presented in the present article,
but for multi-functions, by using a notion of computably discontinuous
multifunction inspired by \cite{BRATTKA2023}?)
\end{rem}

\subsection{\label{subsec:Markov-Conditions}Markov Conditions }

If $(X,\nu)$ is a numbered set, recall that we denote 
\[
A\rightsquigarrow_{\nu}B
\]
 to say that $B$ is not a $\nu$-semi-decidable subset of $A\cup B$. 
\begin{defn}
\label{def:MarkovCond}Let $(X,\nu,\mathcal{T},\tau)$ be a computable
topological space. We say that it satisfies \emph{a Markov condition
with respect to effective closure }if the following holds:
\[
\forall x\in X,\,\forall A\subseteq X,\,x\notin A\,\&\,x\in\overline{A}^{+}\implies A\rightsquigarrow_{\nu}\{x\}.
\]

We say that it satisfies \emph{a Markov condition with respect to
effective sequential closure if }
\[
\forall x\in X,\,\forall A\subseteq X,\,x\notin A\,\&\,x\in\overline{A}^{+\text{seq}}\implies A\rightsquigarrow_{\nu}\{x\}.
\]

We say that it satisfies \emph{a Markov condition with respect to
normed effective sequential closure if }
\[
\forall x\in X,\,\forall A\subseteq X,\,x\notin A\,\&\,x\in\overline{A}^{+\text{seq},N}\implies A\rightsquigarrow_{\nu}\{x\}.
\]
\end{defn}

In words: when a point is effectively adherent to a set (this is a
condition that depends on $(\mathcal{T},\tau)$), this point cannot
be told apart from this set (this is a condition on $\nu$ which does
not depend on the computable topology $(\mathcal{T},\tau)$). 

Obviously the first condition implies the second one, which implies
the third one, since $\overline{A}^{+\text{seq,N}}\subseteq\overline{A}^{+\text{seq}}\subseteq\overline{A}^{+}$. 

Markov conditions can be rendered uniform, by asking that there be
an algorithm that takes as input a program that witnesses for $x\in\overline{A}^{+}$,
and that outputs a reduction to the halting problem: a computable
sequence $(u_{n})_{n\in\mathbb{N}}\in\text{dom}(\nu)^{\mathbb{N}}$
such that: 
\[
\forall n,\,\nu(u_{n})\in A\cup\{x\},
\]
\[
\forall n,\,\nu(u_{n})=x\iff\varphi_{n}(n)\uparrow.
\]

This uniform condition is satisfied in recursive metric spaces that
have passage to the limit algorithms, see Proposition \ref{prop:RMS =00003D> Markov cond}. 

We show that on finite sets, the Ershov topology is the only computable
topology which satisfies a Markov condition. Thus in this case, it
does imply effective continuity of computable functions.
\begin{lem}
\label{lem:Markov on finset}Let $X$ be a finite set equipped with
a numbering $\nu$. Then all notions of effective closure introduced
above coincide, and the only computable topology on $X$ that satisfies
a Markov condition is the Ershov topology. 
\end{lem}

\begin{proof}
In a finite set, it is easy to see that each introduced notion of
effective closure coincides with the actual closure -the non-effective
closure. Indeed, suppose that $x\in\overline{A}$ for a certain subset
$A$ of $X$. Consider the intersection $O_{x}$ of all open sets
of $X$ that contain $x$. This is open, since it is a finite intersection
of open sets. Then $x\in\overline{A}$ if and only if there is a point
$y$ in $A\cap O_{x}$. If there is such a point, the constant sequence
$(y)_{n\in\mathbb{N}}$ is computable and converges to $x$. And thus
$x\in\overline{A}^{+\text{seq}}$, $x\in\overline{A}^{+\text{seq},N}$
and $x\in\overline{A}^{+}$.

Suppose then that a topology $(\mathcal{T},\tau)$ on $(X,\nu)$ satisfies
a Markov condition. Let $A$ be a semi-decidable set, we show that
it is an effective open for $(\mathcal{T},\tau)$, which is thus the
Ershov topology. If $A$ is not open for $(\mathcal{T},\tau)$, there
must be $x$ in $A$, which is adherent to $A^{c}$, thus effectively
adherent to it. And then Markov's condition implies that $\{x\}$
is not a semi-decidable subset of $A^{c}\cup\{x\}$. This contradicts
the fact that $A$ is a semi-decidable subset of $X$. 

We thus have shown that $\tau$ and $\nu_{SD}$ have the same image,
we have yet to prove that these numberings are equivalent to conclude
that the considered topology is the Ershov topology. But this is automatic
in a finite set by the following lemma, Lemma \ref{lem:eff top given by top finset}. 
\end{proof}
\begin{lem}
\label{lem:eff top given by top finset}On a finite numbered set $(X,\nu)$,
given an abstract topology $\mathcal{T}$, there is at most (up to
equivalence) one numbering $\tau$ of $\mathcal{T}$ that makes of
$(X,\nu,\mathcal{T},\tau)$ a computable topological space. This numbering
is the restriction of $\nu_{SD}$ to $\mathcal{T}$. 
\end{lem}

\begin{proof}
Suppose that $(X,\nu,\mathcal{T},\tau)$ is a computable topology.
By definition, $\tau\le\nu_{SD}$, it thus suffices to show the converse
relation for the restriction of $\nu_{SD}$ to $\mathcal{T}$: $\nu_{SD\,\vert\mathcal{T}}\le\tau$.
The set $\mathcal{T}$ is finite, fix a tuple $((A_{1},n_{1}),(A_{2},n_{2}),...,(A_{k},n_{k}))$
that contains each element of $\mathcal{T}$ accompanied by a $\tau$-name
for it. 

Given a $\nu_{SD}$-name for a set $A$ in $\mathcal{T}$, proceed
as follows to obtain a $\tau$-name for it. Enumerate elements of
$A$ thanks to its $\nu_{SD}$-name. At each step of the enumeration,
check if one of the $A_{i}$ is included in the obtained partial enumeration.
If it is, produce the corresponding $\tau$-name $n_{i}$. 

The above procedure gives a $\tau$-computable sequence. Because $(X,\nu,\mathcal{T},\tau)$
is a computable topology, it is possible to compute its union, thus
obtaining a $\tau$-name of $A$. 
\end{proof}
Note that it is unclear whether the Ershov topology on a set $X$
should always satisfy a Markov condition, this is related to the notion
of an intrinsic embedding (see Section \ref{subsec:Intrinsic-embeddings}). 
\begin{problem}
When must the Ershov topology on a numbered set satisfy a Markov condition? 
\end{problem}

\subsection{Examples of spaces that satisfy Markov conditions }

Here, we give two examples of spaces that satisfy Markov conditions:
recursive metric spaces with algorithm of passage to the limit, and
computably sober spaces. Note that a RMS which has an algorithm to
compute limits but which is not computably separable does not have
to be sober. 

The following is a generalization of \cite[Theorem 4.2.2]{Markov1963}.
\begin{prop}
\label{prop:RMS =00003D> Markov cond}A recursive metric space that
admits a passage to the limit algorithm satisfies all Markov conditions
of Definition \ref{def:MarkovCond}. 
\end{prop}

In what follows we denote by $\varphi_{p}(n)\uparrow^{k}$ the fact
that a computation of $\varphi_{p}(n)$ does not end in $k$ steps
(it is thus a relation on triples $(p,n,k)$). 
\begin{proof}
In a RMS, effective closure and effective sequential closure and normed
closure agree, since every point admits a r.e. basis of neighborhoods
(open balls with rational radii centered at this point). Thus it suffices
to prove that Markov's condition is satisfied for the sequential closure. 

Fix a RMS $(X,\nu,d)$, we suppose it admits an algorithm that computes
limits of sequences that converge exponentially fast. 

Suppose that $(u_{n})_{n\in\mathbb{N}}$ is a computable sequence
of $X$ that converges to a computable point $x\in X$. We can in
fact suppose that $(u_{n})_{n\in\mathbb{N}}$ converges at exponential
speed (i.e. that for all $n$, $d(u_{n},x)<2^{-n})$ by passing to
a computable subsequence, because $x$ is computable. 

For each $p$, we define an effective sequence $(w_{n}^{p})_{n\in\mathbb{N}}$
of points in $X$ associated to the $p$th recursive function $\varphi_{p}$.
The sequence $(w_{n}^{p})_{n\in\mathbb{N}}$ is defined as follows:
if $\varphi_{p}(p)\uparrow^{k}$, then $w_{n}^{p}=u_{n}$ for each
$n\le k$. If the computation of $\varphi_{p}(p)$ halts after $k$
steps, $w_{n}^{p}=u_{k}$ for all $n\ge k$. 

Each sequence $(w_{n}^{p})_{n\in\mathbb{N}}$ is Cauchy, and in fact
it converges at least as fast as the original sequence $(u_{n})_{n\in\mathbb{N}}$.
Thus the algorithm of passage to the limit of $X$ can be applied
to any sequence $(w_{n}^{p})_{n\in\mathbb{N}}$. 

Consider the sequence $(w_{p})_{p\in\mathbb{N}}$ obtained by using
the algorithm of passage to the limit on each sequence $(w_{n}^{p})_{n\in\mathbb{N}}$,
for $p\in\mathbb{N}$. This is a $\nu$-computable sequence, since
it was obtained by using the algorithm of passage to the limit which
produces $\nu$-names. 

If follows directly from the construction above that $w_{p}\ne x$
exactly on the halting set $\{p,\,\varphi_{p}(p)\downarrow\}$. And
thus we indeed have $\{u_{n},n\in\mathbb{N}\}\rightsquigarrow_{\nu}\{x\}$. 
\end{proof}
\begin{cor}
\label{prop:Markov COnd N+}$(\mathbb{N}^{+},c_{\mathbb{N}^{+}},\mathcal{T}_{\mathbb{N}^{+}},\tau_{\mathbb{N}^{+}})$
satisfies Markov conditions for all notions of effective closure. 
\end{cor}

\begin{proof}
One easily checks that $(\mathbb{N}^{+},c_{\mathbb{N}^{+}},\mathcal{T}_{\mathbb{N}^{+}},\tau_{\mathbb{N}^{+}})$
is a recursive metric space, with distance inherited from that of
the Cantor Space. It is also easy to see that it has a passage to
the limit algorithm. The result then follows from Proposition \ref{prop:RMS =00003D> Markov cond}. 
\end{proof}
\begin{thm}
\label{thm: Comp Sober =00003D> Markov cond Normed }Let $(X,\nu,\mathcal{T},\tau)$
be a computably sober Type 1 computable topological space. Then $(X,\nu,\mathcal{T},\tau)$
satisfies the Markov condition for the normed effective sequential
closure. 
\end{thm}

\begin{proof}
By Proposition \ref{prop:Comp Seq Normed gives map to Sober space },
any computable normed sequence $(u_{n})_{n\in\mathbb{N}}$ which converges
to a point $x$ in $(X,\nu,\mathcal{T},\tau)$ gives rise to a computable
map $f:\mathbb{N}^{+}\rightarrow X$ given by $f(n)=u_{n}$ and $f(\infty)=x$.
By Corollary \ref{prop:Markov COnd N+}, $\mathbb{N}^{+}$ satisfies
all Markov conditions, so $\{\infty\}$ is not a semi-decidable subset
of $\mathbb{N}^{+}$. Because $f^{-1}(\{x\})=\{\infty\}$, $\{x\}$
cannot be semi-decidable inside of $\{x\}\cup\{u_{n},n\in\mathbb{N}\}$,
as the preimage of a semi-decidable set by a computable map is semi-decidable. 
\end{proof}

\subsection{Relation with the WSO principle }

In this section, we discuss the relation between the Markov conditions
and the constructive principle WSO introduced in \cite{Bauer2012}.
We will use the space $(\mathbb{N}^{+},c_{\mathbb{N}^{+}},\mathcal{T}_{\mathbb{N}^{+}},\tau_{\mathbb{N}^{+}})$
introduced in Section \ref{subsec:A-converging-sequence-and-its-limit}

The following statement, known as the WSO principle, for ``weakly
sequentially open'', was introduced by Bauer and Lešnik in \cite{Bauer2012},
in the context of synthetic topology:
\[
\forall U\in\mathcal{O}(\mathbb{N}^{+}),\,\infty\in U\implies\exists n\in\mathbb{N},\,n\in U.
\]

In \cite{Bauer2012} and in \cite{Bauer2023}, this principle is used
as the key ingredient of various other continuity results.

In terms of Type 1 computability, we can translate this principle
as follows: 
\begin{prop}
\textup{\label{prop: Type 1 WSO }Every $c_{\mathbb{N}^{+}}$-semi-decidable
set that contains $\infty$ also contains a point of $\mathbb{N}$.
Furthermore, there is a computable multi-function that can, given
a $c_{\mathbb{N}^{+}}$-semi-decidable set $A$ that contains $\infty$,
produce a point of $A\cap\mathbb{N}$.}
\end{prop}

\begin{proof}
In $(\mathbb{N}^{+},c_{\mathbb{N}^{+}},\mathcal{T}_{\mathbb{N}^{+}},\tau_{\mathbb{N}^{+}})$,
$\infty$ belongs to the effective closure of $\mathbb{N}$, as it
is the limit of the computable sequence $(n)_{n\in\mathbb{N}}$. By
Corollary \ref{prop:Markov COnd N+}, $(\mathbb{N}^{+},c_{\mathbb{N}^{+}},\mathcal{T}_{\mathbb{N}^{+}},\tau_{\mathbb{N}^{+}})$
satisfies the Markov conditions for the sequential closure, and thus
any semi-decidable set that contains $\infty$ must contain some $n\in\mathbb{N}$.
Such a point can automatically be produced by exhaustive search. 
\end{proof}
The WSO principle is a continuity statement about $\mathbb{N}^{+}$.
However, in a computable topological space $(X,\nu,\mathcal{T},\tau)$,
a computable sequence that converges does not always correspond to
a computable map $\mathbb{N}^{+}\rightarrow X$. And thus this statement
cannot automatically be transferred to other topological spaces, without
using further assumptions. In particular, sobriety is an assumption
that allows to transfer the continuity principle from $(\mathbb{N}^{+},c_{\mathbb{N}^{+}},\mathcal{T}_{\mathbb{N}^{+}},\tau_{\mathbb{N}^{+}})$
to other spaces. See Proposition \ref{prop:Comp Seq Normed gives map to Sober space }. 

The Markov conditions play a very similar role as the WSO principle,
except that they are naturally attached to different spaces, and that
they do not necessarily rely on sobriety- although sober spaces provide,
via Theorem \ref{thm: Comp Sober =00003D> Markov cond Normed }, an
important class of example of spaces that satisfy these conditions.

\subsection{Results on the absence of effective discontinuities }
\begin{lem}
\label{lem: No eff disc 1}Suppose that $(X,\nu,\mathcal{T}_{1},\tau_{1})$
is a computable topological space that satisfies a Markov condition
for the effective closure, and that $(Y,\mu,\mathcal{T}_{2},\tau_{2})$
is a computable topological space. 

Then $(\nu,\mu)$-computable functions between $X$ and $Y$ cannot
have effective discontinuities with respect to $(\mathcal{T}_{1},\tau_{1})$
and $(\mathcal{T}_{2},\tau_{2})$. 
\end{lem}

\begin{proof}
Suppose that $f:X\rightarrow Y$ is a $(\nu,\mu)$-computable function.
Suppose $f$ has an effective discontinuity at $x\in X$. Then there
exists an effective open set $B\subseteq Y$, such that $f(x)\in B$,
but such that $x$ is effectively adherent to the complement of $f^{-1}(B)$. 

The Markov condition for $(X,\nu,\mathcal{T}_{1},\tau_{1})$ then
assures that $\{x\}$ is not a $\nu$-semi-decidable subset of $\{x\}\cup(f^{-1}(B))^{\boldsymbol{c}}$. 

But $f^{-1}(B)$ is a $\nu$-semi-decidable subset of $X$, and thus
$f^{-1}(B)\cap(\{x\}\cup(f^{-1}(B))^{\boldsymbol{c}})$ is a $\nu$-semi-decidable
subset of $\{x\}\cup(f^{-1}(B))^{\boldsymbol{c}}$ (by Proposition
\ref{prop: Semi-decidable set intersected }). But $f^{-1}(B)\cap(\{x\}\cup(f^{-1}(B))^{\boldsymbol{c}})=\{x\}$.

This is a contradiction, and $f$ cannot have an effective discontinuity. 
\end{proof}
Of course, we also have the corresponding proposition for the sequential
closure and the normed sequential closure, we omit its proof. 
\begin{lem}
\label{lem: No eff disc seq 2}Suppose that $(X,\nu,\mathcal{T}_{1},\tau_{1})$
is a computable topological space that satisfies Markov's condition
for the (normed) effective sequential closure, and that $(Y,\mu,\mathcal{T}_{2},\tau_{2})$
is a computable topological space. 

Then $(\nu,\mu)$-computable functions between $X$ and $Y$ cannot
have (normed) effective sequential discontinuities with respect to
$(\mathcal{T}_{1},\tau_{1})$ and $(\mathcal{T}_{2},\tau_{2})$. 
\end{lem}

The hypotheses of this result are weaker, but the result is also weaker:
we prevent less effective discontinuities.

\subsection{Identification of closure and effective closure }

We will now transform the results of absence of effective discontinuities
of Lemma \ref{lem: No eff disc 1} and Lemma \ref{lem: No eff disc seq 2}
into continuity results. 

The known continuity results often use either the effective axiom
of choice of Moschovakis, the stronger hypothesis of existence of
a computable and dense sequence in the domain of the considered function,
or the witness of non-inclusion condition of Spreen. 

We express a different condition here, which asks that the effective
closure of a semi-decidable set correspond to its closure. 
\begin{defn}
Say that a computable topological space $(X,\nu,\mathcal{T},\tau)$
satisfies an \emph{identification condition between the closure and
the effective closure for semi-decidable sets }if for any $\nu$-semi-decidable
set $A$, we have 
\[
\overline{A}=\overline{A}^{+}.
\]
\end{defn}

We define similarly the identification condition between the closure
and the effective (resp. normed effective) sequential closure for
semi-decidable sets, by asking respectively that 
\[
\overline{A}=\overline{A}^{+\text{seq}}
\]
or that 
\[
\overline{A}=\overline{A}^{+\text{seq},N}
\]
 holds for all semi-decidable sets. 

To show that the identification conditions given above are more general
than asking effective separability, we need to generalize a lemma
of Moschovakis. 
\begin{lem}
[Moschovakis, \cite{Moschovakis1964}]\label{lem:sd meet dense moscho}In
a recursive metric space that admits a passage to the limit algorithm,
any dense and computable sequence meets every non-empty semi-decidable
set. 
\end{lem}

Just like the continuity theorems, this result shows that the Ershov
topology and the metric topology are closely related, even though
they can be different.

We generalize it as follows. 
\begin{lem}
In a computable topological space $(X,\nu,\mathcal{T},\tau)$ that
satisfies the Markov condition for the effective closure and that
has a dense and computable sequence, any dense and computable sequence
meets every non-empty semi-decidable set. 
\end{lem}

\begin{proof}
Let $A$ be a non-empty semi-decidable subset of $X$. Denote by $(x_{n})_{n\in\mathbb{N}}$
a dense and computable sequence of $X$. Let $a$ be a point of $A$.
The point $a$ belongs to the effective closure of the sequence $(x_{n})_{n\in\mathbb{N}}$
if, given an open set $O$ that contains $a$, it is possible to compute
a point of $(x_{n})_{n\in\mathbb{N}}$ in $O$. But this is obviously
the case, precisely because $(x_{n})_{n\in\mathbb{N}}$ is a dense
and computable sequence. Thus we can apply Markov's condition: $(x_{n})_{n\in\mathbb{N}}\rightsquigarrow_{\nu}\{a\}$.
Thus there cannot exist a semi-decidable subset of $X$ that contains
$a$ but does not meet the sequence $(x_{n})_{n\in\mathbb{N}}$, in
particular $(x_{n})_{n\in\mathbb{N}}\cap A\ne\emptyset$. 
\end{proof}
As a corollary, we obtain the following result, which shows that our
results do generalize the abstract continuity theorems on effective
Polish spaces. 
\begin{cor}
Suppose that $(X,\nu,\mathcal{T},\tau)$ is a computable topological
space that satisfies the Markov condition for the effective closure
and that has a dense and computable sequence. Then it satisfies the
identification condition between the closure and effective closure
for semi-decidable sets. 
\end{cor}

\begin{proof}
Let $A$ be a semi-decidable subset of $X$, and $x$ be adherent
to $A$. We show that it is effectively adherent to $A$. Let $B$
be a given open neighborhood of $x$. It must meet $A$, so the intersection
$A\cap B$ is a non-empty semi-decidable set, and thus by the previous
lemma it contains a point of the computable and dense sequence on
$X$. Such a point can be found by exhaustive search. 
\end{proof}
To establish the same for the (normed) sequential closure, we need
to ask that points have c.e. bases of neighborhoods. We do not write
it down. 

The identification conditions are expressed for semi-decidable sets.
We could also ask a stronger condition: that closure and effective
closure agree for all sets. This is however not a reasonable requirement,
as it is to restrictive, even if we ask it only for co-semi-decidable
sets. We prove this now. 
\begin{prop}
\label{prop: Identification condition co-semi-decidable sets}Let
$(X,\nu,\mathcal{T},\tau)$ be a computable topological space that
satisfies a Markov condition for the effective closure. Suppose that
for each co-semi-decidable subset $C$ of $X$, $\overline{C}=\overline{C}^{+}$.
Then $\mathcal{T}$ is the abstract Ershov topology on $(X,\nu)$
(i.e. the abstract topology generated by semi-decidable sets).
\end{prop}

\begin{proof}
The proof is similar to that of Lemma \ref{lem:Markov on finset}.
Let $A$ be a $\nu$-semi-decidable set. Suppose that $A$ is not
open for $\mathcal{T}$. Then there is $x$ in $A$ such that $x\in\overline{A^{c}}$.
This implies by hypothesis that $x\in\overline{A^{c}}^{+}$, and thus
that $A^{c}\rightsquigarrow_{\nu}\{x\}$. But this contradicts the
fact that $A$ is semi-decidable. Thus $A$ is open. 
\end{proof}

\subsection{\label{subsec:Multi-numbering-of-neighborhood bases }Multi-numbering
of neighborhood filters and neighborhood bases }

Here we define what it means ``having an effective neighborhood basis
of co-semi-decidable sets'', starting with the multi-numbering of
neighborhoods. 

If $(X,\mathcal{T})$ is a topological space, and $x\in X$, denote
by $\mathcal{N}_{x}$ the neighborhood filter of $x$. 
\begin{defn}
Let $(X,\nu,\mathcal{T},\tau)$ be a computable topological space.
For each point $x$ of $X$, define a multi-numbering $\mathfrak{n}_{x}:\mathbb{N}\rightrightarrows\mathcal{N}_{x}$
of neighborhoods of $x$ by 
\[
\text{dom}(\mathfrak{n}_{x})=\{n\in\text{dom}(\tau),\,x\in\tau(n)\};
\]
\[
\forall n\in\text{dom}(\mathfrak{n}_{x}),\,\mathfrak{n}_{x}(n)=\{A\subseteq X,\,\tau(n)\subseteq A\}.
\]
\end{defn}

Thus a neighborhood $A$ of $x$ is described by an open set that
contains $x$ and that is contained in $A$. Of course, this does
not define $A$ uniquely, and $\mathfrak{n}_{x}$ is a multi-numbering.
All neighborhoods of $x$ have a description for $\mathfrak{n}_{x}$. 

The fact that the set $\mathcal{N}_{x}$ is a filter is mirrored in
the properties of $\mathfrak{n}_{x}$ in the following way: 
\begin{itemize}
\item A filter on $X$ always contains $X$, and indeed for all $n$ in
the domain of $\mathfrak{n}_{x}$, $X\in\mathfrak{n}_{x}(n)$. 
\item A filter is stable under finite intersections, and the intersection
map is $(\mathfrak{n}_{x}\times\mathfrak{n}_{x},\mathfrak{n}_{x})$-computable
for neighborhoods of $x$. 
\item A filter is stable under supersets, and here, if $n$ is a $\mathfrak{n}_{x}$-name
of a neighborhood $A$ of $x$ and if $B$ is any subset of $X$ such
that $A\subseteq B$, then in fact $n$ is also a $\mathfrak{n}_{x}$-name
of $B$. 
\end{itemize}
Note that this last point provides a nice solution to the problem
of rendering effective the statement, for a filter $\mathcal{F}$
of $X$: 
\[
\forall A\in\mathcal{F},\,\forall B\subseteq X,\,A\subseteq B\implies B\in\mathcal{F}.
\]
Indeed, when defining a computable filter $\mathcal{F}$, we can expect
that a numbering will be associated to $\mathcal{F}$, that gives
descriptions of the elements of $\mathcal{F}$, but we cannot expect
to have a canonical numbering of all subsets of $X$, this makes it
hard to find an effective statement associated to the statement above. 

The use of the multi-numberings $\mathfrak{n}_{x}$ in fact trivializes
the statement highlighted above. 

Classically, a neighborhood basis for a point $x$ is a set $\mathcal{B}_{x}\subseteq\mathcal{P}(X)$
such that for any neighborhood $A$ of $x$, there exists $B$ in
$\mathcal{B}_{x}$ such that $B\subseteq A$. 
\begin{defn}
\emph{\label{def: Effective-neighborhood-basis }An effective neighborhood
basis }for a point $x$ is a pair $(\mathcal{B}_{x},\beta_{x})$,
where $\mathcal{B}_{x}$ is a neighborhood basis of $x$, and $\beta_{x}$
is a numbering of $\mathcal{B}_{x}$, such that there is a procedure
that, given the $\mathfrak{n}_{x}$-name of a neighborhood $A$ of
$x$, produces the $\beta_{x}$-name of an element $B$ with $B\subseteq A$. 
\end{defn}

Note that the $\mathfrak{n}_{x}$-name of $A$ does not characterize
it uniquely, so the element $B$ in the definition above should in
fact be contained in all neighborhoods of $x$ that are described
by this name. 

We could relax the above definition to allow $\beta_{x}$ to be a
multi-numbering, but we will not need this more general definition
in the present article. 
\begin{example}
The set of effective open sets that contain $x$ together with the
numbering $\tau$ trivially form an effective neighborhood basis of
$x$: if $n$ is a $\mathfrak{n}_{x}$-name of a neighborhood $A$,
then $n$ is in fact the $\tau$-name of an open set contained in
$A$, so the procedure required for Definition \ref{def: Effective-neighborhood-basis }
is just the identity. 
\end{example}

We can now finally state the definition that will allow us to obtain
continuity results: 
\begin{defn}
A point $x$ in a computable topological space $(X,\nu,\mathcal{T},\tau)$
has \emph{an effective neighborhood basis of co-semi-decidable sets
}if there exists an effective neighborhood basis $(\mathcal{B}_{x},\beta_{x})$
for $x$ such that for all $n$ in $\text{dom}(\beta_{x})$, $\beta_{x}(n)$
is a co-semi-decidable set. 

A point $x$ in a computable topological space $(X,\nu,\mathcal{T},\tau)$
has \emph{a neighborhood basis of uniformly co-semi-decidable sets
}if there exists an effective neighborhood basis $(\mathcal{B}_{x},\beta_{x})$
for $x$ such that $\beta_{x}\le\nu_{coSD}$. 
\end{defn}

Finally, we define neighborhood bases that are uniform in $x$. 
\begin{defn}
Points in $X$ \emph{uniformly have bases of uniformly co-semi-decidable
sets} if each computable point $x$ admits a basis $(\mathcal{B}_{x},\beta_{x})$
of uniformly co-semi-decidable sets, and if the two procedures associated
to $(\mathcal{B}_{x},\beta_{x})$ depend recursively on the parameter
$x$: 
\begin{itemize}
\item There is a procedure that takes as input the $\nu$-name of a point
$x$ and the $\mathfrak{n}_{x}$-name of a neighborhood $A$ of $x$,
and produces the $\beta_{x}$-name of a set $B\in\mathcal{B}_{x}$
with $B\subseteq A$;
\item There is a procedure that takes as input the $\nu$-name of a point
$x$ and the $\beta_{x}$-name of a set $B\in\mathcal{B}_{x}$, and
produces a $\nu_{coSD}$-name for $B$. 
\end{itemize}
\end{defn}

It is easy to see that in recursive metric spaces,\emph{ }points\emph{
}uniformly have bases of uniformly co-semi-decidable sets: closed
balls are co-semi-decidable sets. 

\bigskip

\subsection{Continuity results }

The identification conditions are used for the following lemma: 
\begin{lem}
\label{lem: Disc + identification -> eff disc }Suppose that $(X,\nu,\mathcal{T}_{1},\tau_{1})$
is a computable topological space that satisfies the identification
condition between the closure and the effective closure (resp. (normed)
effective sequential closure) for semi-decidable sets, and that $(Y,\mu,\mathcal{T}_{2},\tau_{2})$
is a computable topological space where points have co-semi-decidable
bases of neighborhood. 

If a $(\nu,\mu)$-computable function $f:X\rightarrow Y$ has a discontinuity
at a point, it also has an effective discontinuity at this point (resp.
a (normed) effective sequential discontinuity). 
\end{lem}

\begin{proof}
Suppose $f$ has a discontinuity at a point $x$, i.e. we have an
open set $B$ of $Y$ which contains $f(x)$, but such that $x$ is
adherent to $f^{-1}(B)^{\boldsymbol{c}}$. Using the hypothesis on
$Y$, we can replace $B$ by a smaller set $C$, which is also a neighborhood
of $f(x)$, but which is co-semi-decidable. 

Then we also have that $x$ belongs to the closure of $f^{-1}(C)^{\boldsymbol{c}}$,
and since this set is semi-decidable, we can apply the identification
hypothesis, which gives the desired effective discontinuity. 
\end{proof}
We finally put together the non-effective discontinuity results with
the identification conditions to obtain abstract continuity results. 
\begin{thm}
[Continuity result for the Markov condition for effective closure]\label{thm:Continuity-result-eff-closure}Suppose
that $(X,\nu,\mathcal{T}_{1},\tau_{1})$ is a computable topological
space that satisfies Markov's condition for the effective closure
and the identification condition between the closure and the effective
closure for semi-decidable sets. 

Suppose that $(Y,\mu,\mathcal{T}_{2},\tau_{2})$ is a computable topological
space where points have co-semi-decidable bases of neighborhood. 

Then $(\nu,\mu)$-computable functions between $X$ and $Y$ are $(\mathcal{T}_{1},\mathcal{T}_{2})$-continuous. 
\end{thm}

\begin{proof}
This is just a joint application of Lemma \ref{lem: No eff disc 1}
and Lemma \ref{lem: Disc + identification -> eff disc }. 
\end{proof}
\begin{thm}
[Continuity result for the Markov condition for effective sequential
closure]\label{thm:Continuity-result-eff-closure-SEQ 2}Suppose that
$(X,\nu,\mathcal{T}_{1},\tau_{1})$ is a computable topological space
that satisfies Markov's condition for the effective sequential closure
and the identification condition between the closure and the effective
sequential closure for semi-decidable sets. 

Suppose that $(Y,\mu,\mathcal{T}_{2},\tau_{2})$ is a computable topological
space where points have co-semi-decidable bases of neighborhood. 

Then $(\nu,\mu)$-computable functions between $X$ and $Y$ are $(\mathcal{T}_{1},\mathcal{T}_{2})$-continuous.
\end{thm}

\begin{proof}
This is just a joint application of Lemma \ref{lem: No eff disc seq 2}
and Lemma \ref{lem: Disc + identification -> eff disc }. 
\end{proof}
\begin{thm}
[Continuity result for the Markov condition for normed effective
sequential closure]\label{thm:Continuity-result-eff-closure-SEQ 2-1}Suppose
that $(X,\nu,\mathcal{T}_{1},\tau_{1})$ is a computable topological
space that satisfies Markov's condition for the normed effective sequential
closure and the identification condition between the closure and the
normed effective sequential closure for semi-decidable sets. 

Suppose that $(Y,\mu,\mathcal{T}_{2},\tau_{2})$ is a computable topological
space where points have co-semi-decidable bases of neighborhood. 

Then $(\nu,\mu)$-computable functions between $X$ and $Y$ are $(\mathcal{T}_{1},\mathcal{T}_{2})$-continuous.
\end{thm}

\begin{proof}
This is just a joint application of Lemma \ref{lem: No eff disc seq 2}
and Lemma \ref{lem: Disc + identification -> eff disc }. 
\end{proof}
These theorem are not \emph{strictly more general} than previously
known results, as they of course do not allow to recover the effective
continuity theorems of Ceitin, Spreen, Moschovakis, etc, but, as stated,
they are not consequences of known results either, and thus are some
of the most general continuity theorems that exist for computable
functions. 

\subsection{\label{subsec:Banach-Mazur-computability}Banach-Mazur computability}

In this section, we note that our results that rely on sequential
continuity also apply to Banach-Mazur computability. 

Recall the definition of a Banach-Mazur computable function (\cite{Mazur1963}
for real-valued functions, the general definition was given by Hertling
in \cite{Hertling2001}):
\begin{defn}
Let $(X,\nu)$ and $(Y,\mu)$ be numbered sets. A function $f:X\rightarrow Y$
is called \emph{Banach-Mazur $(\nu,\mu)$-computable}, or \emph{BM
$(\nu,\mu)$-computable}, if for every $\nu$-computable sequence
$(u_{n})_{n\in\mathbb{N}}\in X^{\mathbb{N}}$, the sequence $(f(u_{n}))_{n\in\mathbb{N}}$
is $\mu$-computable. 
\end{defn}

Working with Banach-Mazur computability, the relevant notion of effective
closure is the sequential closure. 
\begin{defn}
[Markov condition for Banach-Mazur computability]Let $(X,\nu,\mathcal{T},\tau)$
be a computable topological space. We says that it satisfies a \emph{Markov
condition for Banach-Mazur computable} \emph{functions} if, whenever
$x\in\overline{A}^{+\text{seq}}$, there exists a $\nu$-computable
sequence $(y_{n})_{n\in\mathbb{N}}\in(\{x\}\cup A)^{\mathbb{N}}$
such that $\{n\in\mathbb{N},\,y_{n}=x\}$ is not a c.e. set. 
\end{defn}

Note that the condition above is stronger than simply asking that
$\{x\}$ is not $\nu$-semi-decidable inside $\{x\}\cup A$, in terms
of Markov computability: it says that there is a c.e. set of $\nu$-names
on which the problem is already undecidable. In terms of Markov computability,
the only relevant information is whether or not the problem is decidable
on the whole set of names of points in $\{x\}\cup A$, this set is
in general never c.e. (for instance, in the computable reals, the
set of all names of a single computable real is not c.e.). 

The following theorem is derived exactly as Theorem \ref{thm:Continuity-result-eff-closure-SEQ 2},
we omit its proof:
\begin{thm}
[Banach-Mazur Continuity result]\label{thm:Banach-Mazur-Continuity-result}Suppose
that $(X,\nu,\mathcal{T}_{1},\tau_{1})$ is a computable topological
space that satisfies Markov's condition for Banach-Mazur computable
functions and the identification condition between the closure and
the effective sequential closure for semi-decidable sets. 

Suppose that $(Y,\mu,\mathcal{T}_{2},\tau_{2})$ is a computable topological
space where points have co-semi-decidable bases of neighborhood. 

Then Banach-Mazur $(\nu,\mu)$-computable functions between $X$ and
$Y$ are $(\mathcal{T}_{1},\mathcal{T}_{2})$-continuous.
\end{thm}

Hertling has shown in \cite{Hertling2002} that Banach-Mazur computable
functions defined on the computable reals, while always continuous,
do not have to be effectively continuous. 

This result can be used as a sort of reverse mathematics tool: using
solely assumptions that are true of Banach-Mazur computable function,
one can never obtain an effective continuity result. 

This indicates that in order to prove an effective continuity result
basing ourselves on Markov conditions, we have to introduce arguments
that are significantly new compared with those that are used in the
proof of Theorem \ref{thm:Banach-Mazur-Continuity-result} above (and
also compared with those of Theorems \ref{thm:Continuity-result-eff-closure},
\ref{thm:Continuity-result-eff-closure-SEQ 2} and \ref{thm:Continuity-result-eff-closure-SEQ 2-1}).
But this does not imply that such an effective continuity result cannot
be obtained under the hypotheses that we present in the next section,
Section \ref{subsec:Effective-Continuity-conjecture}. Indeed, Banach-Mazur
computable functions defined on a recursive Polish space are continuous,
and Markov computable functions defined there are effectively continuous:
the fact that under certain hypotheses we can prove a continuity result
for Banach-Mazur computable functions does not imply that we cannot
prove an effective continuity result for Markov computable functions
under those same hypotheses. However, the proof should incorporate
significantly new arguments.

\subsection{\label{subsec:Effective-Continuity-conjecture}Effective Continuity
conjecture }

Each of the conditions stated to prove Theorem \ref{thm:Continuity-result-eff-closure}
and Theorem \ref{thm:Continuity-result-eff-closure-SEQ 2} can be
rendered uniformly. It is only after we have defined such uniform
conditions that we can expect to prove effective continuity results. 

There are three hypotheses that need to be expressed uniformly: 
\begin{itemize}
\item The uniform version of Markov's condition, which was already quoted
in Section \ref{subsec:Markov-Conditions}. 
\item A uniform version of the identification condition between closure
and effective closure, which we detail below. 
\item And the fact that points in the codomain of the considered function
uniformly have neighborhood bases of uniformly co-semi-decidable sets,
this was already defined in Section \ref{subsec:Multi-numbering-of-neighborhood bases }. 
\end{itemize}
The uniform version of the identification between the closure and
effective closure for semi-decidable sets can be stated as follows:
\begin{defn}
Say that a computable topological space $(X,\nu,\mathcal{T},\tau)$
satisfies \emph{the uniform} \emph{identification condition between
the closure and the effective closure for semi-decidable sets }if
there is an algorithm that, given the $\nu$-name of a point $x$
and a $\nu_{SD}$-name of a semi-decidable set $A$, such that $x\in\overline{A}$,
produces the code for a function that testifies for the relation $x\in\overline{A}^{+}$,
i.e. a computable map that, given the $\tau$-name of an open set
$O$ which contains $x$, produces the $\nu$-name of a point in $O\cap A$. 
\end{defn}

Using curryfication, one can see easily that the above condition exactly
asks the following: that there be a $(\nu_{SD}\times\nu\times\tau,\nu)$-computable
multi-function $P$ (i.e. the result of $P$ is allowed to depend
on the given names for its input) such that, for all $A$ semi-decidable,
$x$ computable point, and $O$ effective open set, we have: 
\[
P(A,x,O)\subseteq A\cap O.
\]

This map can be computed under the hypothesis that $x$ is adherent
to $A$, and that $x$ belongs to $O$ (in which case $A\cap O$ is
necessarily non-empty). 

Compare then this to Moschovakis' Effective Choice Axiom condition,
Condition (B) in \cite{Moschovakis1964}. (We slightly reformulate
it so that it does not make reference to a basis, but this is easily
seen to be equivalent to Moschovakis' condition when the topology
comes from a basis.) A computable topological space $(X,\nu,\mathcal{T},\tau)$
satisfies \emph{Moschovakis' effective choice axiom} if there is a
$(\nu_{SD}\times\nu\times\tau,\nu)$-computable multi-function $C$
such that, for any triple $(A,x,O)$, where $A$ is a semi-decidable
set, $x$ a computable point, and $O$ an open set that contains $x$,
we have: 
\[
C(A,x,O)\subseteq A\cap O.
\]

The \emph{only difference} with what is above is that this map can
be computed under the sole assumption that $A\cap O\ne\emptyset$. 

Thus Moschovakis' condition is more restrictive than ``effective
identification between closure and effective closure for semi-decidable
sets'', because it asks for a program that, while it should perform
exactly the same computation as $P$, it should be able to do it on
a wider set of inputs. 
\begin{problem}
Can we obtain an effective continuity theorem that generalizes Moschovakis'
theorem using the effective identification between closure and effective
closure for semi-decidable sets condition? 
\end{problem}

Obtaining this even in an effectively complete recursive metric space
would be interesting. 

\bibliographystyle{alpha}
\bibliography{BiblioContinuityMarkovCond}

\begin{thebibliography}{{Bor}12}

\bibitem[Bau00]{Bauer2000}
Andrej Bauer.
\newblock {\em The realizability approach to computable analysis and topology}.
\newblock PhD thesis, School of Computer Science, Carnegie Mellon University,
  2000.

\bibitem[Bau23]{Bauer2023}
Andrej Bauer.
\newblock Spreen spaces and the synthetic
  {K}reisel-{L}acombe-{S}hoenfield-{T}seitin theorem.
\newblock {\em 10.48550/arxiv.2307.07830}, July 2023.

\bibitem[Bee76]{Beeson1976}
Michael Beeson.
\newblock The unprovability in intuitionistic formal systems of the continuity
  of effective operations on the reals.
\newblock {\em The Journal of Symbolic Logic}, 41(1):18--24, March 1976.

\bibitem[BH21]{Brattka2021}
Vasco Brattka and Peter Hertling, editors.
\newblock {\em Handbook of Computability and Complexity in Analysis}.
\newblock Springer International Publishing, 2021.

\bibitem[BL12]{Bauer2012}
Andrej Bauer and Davorin Le{\v s}nik.
\newblock Metric spaces in synthetic topology.
\newblock {\em Annals of Pure and Applied Logic}, 163(2):87--100, February
  2012.

\bibitem[BM37]{Banach1937}
Stefan Banach and Stanislaw Mazur.
\newblock Sur les fonctions calculables.
\newblock {\em Ann. Soc. Pol. de Math}, 16(223):402, 1937.

\bibitem[{Bor}12]{Borel1912}
Emile {Borel}.
\newblock Le calcul des int{\'e}grales d{\'e}finies.
\newblock {\em Journal de Math{\'e}matiques Pures et Appliqu{\'e}es. 6.
  S{\'e}rie}, 8:159--210, 1912.

\bibitem[Bra23]{BRATTKA2023}
Vasco Brattka.
\newblock The discontinuity problem.
\newblock {\em The Journal of Symbolic Logic}, 88(3):1191--1212, 2023.

\bibitem[Cei67]{Ceitin1967}
Gregory~S. Ceitin.
\newblock Algorithmic operators in constructive metric spaces.
\newblock {\em Trudy Matematicheskogo Instituta Imeni V. A. Steklova}, pages
  1--80, 1967.

\bibitem[Cut97]{Cutland1997}
Nigel Cutland.
\newblock {\em Computability}.
\newblock Cambridge Univ. Press, Cambridge [u.a.], reprinted edition, 1997.
\newblock Literaturverz. S. [239] - 240.

\bibitem[Fri58]{Friedberg1958}
Richard Friedberg.
\newblock Un contre-exemple relatif aux fonctionnelles r{\'e}cursives.
\newblock {\em Comptes rendus hebdomadaires des s{\'e}ances de l'Acad{\'e}mie
  des Sciences (Paris)}, vol. 24, 1958.

\bibitem[Her96]{Hertling1996}
Peter Hertling.
\newblock Computable real functions: Type 1 computability versus type 2
  computability.
\newblock In Ker{-}I Ko, Norbert~Th. M{\"{u}}ller, and Klaus Weihrauch,
  editors, {\em Second Workshop on Computability and Complexity in Analysis,
  {CCA} 1996, August 22-23, 1996, Trier, Germany}, volume {TR} 96-44 of {\em
  Technical Report}. Unjiversity of Trier, 1996.

\bibitem[Her01]{Hertling2001}
Peter Hertling.
\newblock {B}anach-{M}azur computable functions on metric spaces.
\newblock In {\em Computability and Complexity in Analysis}, pages 69--81.
  Springer Berlin Heidelberg, 2001.

\bibitem[Her02]{Hertling2002}
Peter Hertling.
\newblock A {Banach-Mazur} computable but not {Markov} computable function on
  the computable real numbers.
\newblock In {\em Automata, Languages and Programming}, volume 132, pages
  962--972. Springer Berlin Heidelberg, mar 2002.

\bibitem[HR16]{Hoyrup2016}
Mathieu Hoyrup and Crist{\'{o}}bal Rojas.
\newblock On the information carried by programs about the objects they
  compute.
\newblock {\em Theory of Computing Systems}, 61(4):1214--1236, dec 2016.

\bibitem[Ish91]{Ishihara1991}
Hajime Ishihara.
\newblock Continuity and nondiscontinuity in constructive mathematics.
\newblock {\em J. Symb. Log.}, 56(4):1349--1354, 1991.

\bibitem[Ish92]{Ishihara1992}
Hajime Ishihara.
\newblock Continuity properties in constructive mathematics.
\newblock {\em J. Symb. Log.}, 57(2):557--565, 1992.

\bibitem[KLS57]{Kreisel1957}
Georg Kreisel, Daniel Lacombe, and Joseph~R. Shoenfield.
\newblock Partial recursive functionals and effective operations.
\newblock Constructivity in mathematics, Proceedings of the colloquium held at
  Amsterdam:pp. 290--297, 1957.

\bibitem[Kus84]{Kushner1984}
Boris~A. Kushner.
\newblock {\em Lectures on Constructive Mathematical Analysis}.
\newblock American Mathematical Society, 1984.

\bibitem[Mar63]{Markov1963}
Andrei~Andreevich Markov.
\newblock {\em On constructive functions}, chapter {i}n Twelve papers on logic
  and differential equations, pages 163--195.
\newblock American Mathematical Society Translations, 1963.

\bibitem[Maz63]{Mazur1963}
Stanislaw Mazur.
\newblock {\em Computable analysis}.
\newblock 1963.
\newblock http://eudml.org/doc/268535.

\bibitem[Mos64]{Moschovakis1964}
Yiannis Moschovakis.
\newblock Recursive metric spaces.
\newblock {\em Fundamenta Mathematicae}, 55(3):215--238, 1964.

\bibitem[MS55]{Myhill1955}
John Myhill and John~C. Shepherdson.
\newblock Effective operations on partial recursive functions.
\newblock {\em Zeitschrift f{\"u}r Mathematische Logik und Grundlagen der
  Mathematik}, 1(4):310--317, 1955.

\bibitem[Rau21]{Rauzy2021}
Emmanuel Rauzy.
\newblock Computable analysis on the space of marked groups.
\newblock {\em arXiv:2111.01179}, 2021.

\bibitem[Rau23]{RAUZY2023}
Emmanuel Rauzy.
\newblock New definitions in the theory of type 1 computable topological
  spaces.
\newblock {\em arXiv:2311.16340}, 2023.

\bibitem[Sch03]{Schroeder2003}
Matthias Schr{\"o}der.
\newblock Admissible representations for continuous computations.
\newblock 2003.

\bibitem[Spr90]{Spreen1990}
Dieter Spreen.
\newblock A characterization of effective topological spaces.
\newblock In {\em Lecture Notes in Mathematics}, pages 363--387. Springer
  Berlin Heidelberg, 1990.

\bibitem[Spr98]{Spr98}
Dieter Spreen.
\newblock On effective topological spaces.
\newblock {\em The Journal of Symbolic Logic}, 63(1):185--221, 1998.

\bibitem[Spr10]{Spreen_2010}
Dieter Spreen.
\newblock Effectivity and effective continuity of multifunctions.
\newblock {\em The Journal of Symbolic Logic}, 75(2):602--640, jun 2010.

\bibitem[Spr16]{SPREEN_2016}
Dieter Spreen.
\newblock Some results related to the continuity problem.
\newblock {\em Mathematical Structures in Computer Science}, 27(8):1601--1624,
  jun 2016.

\bibitem[SY84]{Spreen1984}
Dieter Spreen and Paul Young.
\newblock Effective operators in a topological setting.
\newblock In {\em Computation and Proof Theory}, pages 437--451. Springer
  Berlin Heidelberg, 1984.

\bibitem[Tay02]{Taylor2002}
Paul Taylor.
\newblock Sober spaces and continuations.
\newblock {\em Theory and Applications of Categories}, 10(12):248--299, July
  2002.

\bibitem[Tur37]{Turing1937}
Alan~M. Turing.
\newblock On computable numbers, with an application to the
  {E}ntscheidungsproblem.
\newblock {\em Proceedings of the London Mathematical Society},
  s2-42(1):230--265, 1937.

\bibitem[Tur38]{Turing1938}
Alan~M. Turing.
\newblock On computable numbers, with an application to the
  {E}ntscheidungsproblem. {A} correction.
\newblock {\em Proceedings of the London Mathematical Society},
  s2-43(1):544--546, 1938.

\end{thebibliography}

\end{document}